\documentclass[reqno]{amsart}

\usepackage{fullpage,amssymb,dsfont}

\usepackage{tikz-cd}

\usepackage[nobysame,alphabetic,initials,msc-links,lite]{amsrefs}

\DefineSimpleKey{bib}{how}
\renewcommand{\PrintDOI}[1]{%
  \href{http://dx.doi.org/#1}{{\tt DOI:#1}}%
}
\renewcommand{\eprint}[1]{#1}
\BibSpec{book}{%
    +{}  {\PrintPrimary}                {transition}
    +{.} { \PrintDate}                  {date}
    +{.} { \textit}                     {title}
    +{.} { }                            {part}
    +{:} { \textit}                     {subtitle}
    +{,} { \PrintEdition}               {edition}
    +{}  { \PrintEditorsB}              {editor}
    +{,} { \PrintTranslatorsC}          {translator}
    +{,} { \PrintContributions}         {contribution}
    +{,} { }                            {series}
    +{,} { \voltext}                    {volume}
    +{,} { }                            {publisher}
    +{,} { }                            {organization}
    +{,} { }                            {address}
    +{,} { }                            {status}
    +{,} { \PrintDOI}                   {doi}
    +{,} { \PrintISBNs}                 {isbn}
    +{}  { \parenthesize}               {language}
    +{}  { \PrintTranslation}           {translation}
    +{;} { \PrintReprint}               {reprint}
    +{.} { }                            {note}
    +{.} {}                             {transition}
    +{}  {\SentenceSpace \PrintReviews} {review}
}
\BibSpec{article}{%
    +{}  {\PrintAuthors}                {author}
    +{,} { \textit}                     {title}
    +{.} { }                            {part}
    +{:} { \textit}                     {subtitle}
    +{,} { \PrintContributions}         {contribution}
    +{.} { \PrintPartials}              {partial}
    +{,} { }                            {journal}
    +{}  { \textbf}                     {volume}
    +{}  { \PrintDatePV}                {date}
    +{,} { \issuetext}                  {number}
    +{,} { \eprintpages}                {pages}
    +{,} { }                            {status}
    +{,} { \PrintDOI}                   {doi}
    +{,} { \eprint}        {eprint}
    +{}  { \parenthesize}               {language}
    +{}  { \PrintTranslation}           {translation}
    +{;} { \PrintReprint}               {reprint}
    +{.} { }                            {note}
    +{.} {}                             {transition}
    +{}  {\SentenceSpace \PrintReviews} {review}
}
\BibSpec{collection.article}{%
    +{}  {\PrintAuthors}                {author}
    +{,} { \textit}                     {title}
    +{.} { }                            {part}
    +{:} { \textit}                     {subtitle}
    +{,} { \PrintContributions}         {contribution}
    +{,} { \PrintConference}            {conference}
    +{}  {\PrintBook}                   {book}
    +{,} { }                            {booktitle}
    +{,} { \PrintDateB}                 {date}
    +{,} { pp.~}                        {pages}
    +{,} { }                            {publisher}
    +{,} { }                            {organization}
    +{,} { }                            {address}
    +{,} { }                            {status}
    +{,} { \PrintDOI}                   {doi}
    +{,} { \eprint}        {eprint}
    +{}  { \parenthesize}               {language}
    +{}  { \PrintTranslation}           {translation}
    +{;} { \PrintReprint}               {reprint}
    +{.} { }                            {note}
    +{.} {}                             {transition}
    +{}  {\SentenceSpace \PrintReviews} {review}
}
\BibSpec{misc}{%
  +{}{\PrintAuthors}  {author}
  +{,}{ \textit}      {title}
  +{.}{ }             {how}
  +{}{ \parenthesize} {date}
  +{,} { available at \eprint}        {eprint}
  +{,}{ available at \url}{url}
  +{,} { \PrintDOI}                   {doi}
  +{,}{ }             {note}
  +{.}{}              {transition}
}

\numberwithin{equation}{section}

\newtheorem{theorem}{Theorem}[section]
\newtheorem{corollary}[theorem]{Corollary}
\newtheorem{lemma}[theorem]{Lemma}
\newtheorem{proposition}[theorem]{Proposition}

\theoremstyle{remark}
\newtheorem{remark}[theorem]{Remark}

\newtheorem{example}[theorem]{Example}
\theoremstyle{definition}
\newtheorem{definition}[theorem]{Definition}

\newcommand{\bp}{\begin{proof}}
\newcommand{\ep}{\end{proof}}

\DeclareFontFamily{U}{matha}{\hyphenchar\font45}
\DeclareFontShape{U}{matha}{m}{n}{
      <5> <6> <7> <8> <9> <10> gen * matha
      <10.95> matha10 <12> <14.4> <17.28> <20.74> <24.88> matha12
      }{}
\DeclareSymbolFont{matha}{U}{matha}{m}{n}
\DeclareFontFamily{U}{mathb}{\hyphenchar\font45}
\DeclareFontShape{U}{mathb}{m}{n}{
      <5> <6> <7> <8> <9> <10> gen * mathb
      <10.95> mathb10 <12> <14.4> <17.28> <20.74> <24.88> mathb12
      }{}
\DeclareSymbolFont{mathb}{U}{mathb}{m}{n}
\DeclareMathSymbol{\ovoid}{3}{matha}{"6C}
\DeclareMathSymbol{\boxvoid}      {2}{mathb}{"6C}

\usepackage{wasysym}% ocircle

\newcommand{\circt}%
{\mathbin{%
\mathchoice
{\ooalign{$\ocircle$\cr\hidewidth\raise-.15ex\hbox{$\scriptstyle\top\mkern1.7mu$}\cr}}% Woronowicz style tensor product, USUAL SIZE
{\ooalign{$\ocircle$\cr\hidewidth\raise-.15ex\hbox{$\scriptstyle\top\mkern1.7mu$}\cr}}% Woronowicz style tensor product, USUAL SIZE
{\ooalign{$\scriptstyle\ocircle$\cr\hidewidth\raise-.12ex\hbox{$\scriptscriptstyle\top\mkern1mu$}\cr}}% Woronowicz style tensor product, SCRIPT SIZE
{\ooalign{$\scriptstyle\ocircle$\cr\hidewidth\raise-.12ex\hbox{$\scriptscriptstyle\top\mkern1mu$}\cr}}% Woronowicz style tensor product, SCRIPT SIZE
}}

\mathchardef\mhyph="2D

\newcommand\Mod{\mathrm{Mod}}
\newcommand{\C}{\mathbb{C}}

\newcommand\A{\mathcal A}
\newcommand{\CC}{\mathcal{C}}
\newcommand{\BB}{\mathcal{B}}
\newcommand\DD{\mathcal{D}}

\newcommand\U{\mathcal U}

\newcommand{\un}{\mathds{1}}
\newcommand\Hilbf{{\Hilb}_f}
\newcommand{\cA}{\mathcal{A}}

\newcommand{\cC}{\mathcal{C}}
\newcommand{\cD}{\mathcal{D}}
\newcommand{\cE}{\mathcal{E}}

\newcommand{\cX}{\mathcal{X}}
\newcommand{\cY}{\mathcal{Y}}

\newcommand{\norm}[1]{\left\|#1\right\|}

\newcommand\op{\mathrm{op}}
\newcommand{\opos}{\mathrm{op}}

\newcommand{\Hilb}{\mathrm{Hilb}}
\newcommand{\Id}{\mathrm{Id}}

\DeclareMathOperator{\Ad}{Ad}
\DeclareMathOperator\BrPic{BrPic}
\DeclareMathOperator{\Corr}{Corr}

\DeclareMathOperator{\Bimod}{Bimod}
\DeclareMathOperator{\Irr}{Irr}

\DeclareMathOperator{\Mat}{Mat}
\DeclareMathOperator{\Mor}{Mor}
\DeclareMathOperator{\End}{End}
\DeclareMathOperator{\Hom}{Hom}
\DeclareMathOperator{\inHom}{\underline{Hom}}
\DeclareMathOperator{\inEnd}{\underline{End}}
\DeclareMathOperator{\Tr}{Tr}
\DeclareMathOperator{\tr}{tr}
\DeclareMathOperator{\Rep}{Rep}

\newcommand{\lmodcat}[1]{{#1}\mhyph\mathrm{Mod}}
\newcommand{\rmodcat}[1]{\mathrm{Mod}\mhyph{#1}}
\newcommand{\lmodcatin}[2]{{#1}\mhyph\mathrm{Mod}_{#2}}
\newcommand{\rmodcatin}[2]{\mathrm{Mod}_{#2}\mhyph{#1}}
\newcommand{\bimodcat}[1]{\Bimod\mhyph{#1}}
\newcommand{\bimodcatin}[2]{\Bimod_{#2}\mhyph{#1}}

\newcommand{\Amod}{\lmodcat{A}}
\newcommand{\modA}{\rmodcat{A}}

\newcommand{\modQin}[1]{\rmodcatin{Q}{#1}}

\theoremstyle{remark}

\begin{document}

\title{Categorically Morita equivalent compact quantum groups}

\author[S. Neshveyev]{Sergey Neshveyev}

\email{sergeyn@math.uio.no}

\address{Department of Mathematics, University of Oslo, P.O. Box 1053
Blindern, NO-0316 Oslo, Norway}

\thanks{Supported by the European Research Council under the European Union's Seventh
Framework Programme (FP/2007-2013)
/ ERC Grant Agreement no. 307663%--NCGQG
}

\author[M. Yamashita]{Makoto Yamashita}

\email{yamashita.makoto@ocha.ac.jp}

\address{Department of Mathematics, Ochanomizu University, Otsuka
2-1-1, Bunkyo, 112-8610 Tokyo, Japan}

\begin{abstract}
We give a dynamical characterization of categorical Morita
equivalence between compact quantum groups. More precisely, by a
Tannaka--Krein type duality, a unital C$^*$-algebra endowed with
commuting actions of two compact quantum groups corresponds to a
bimodule category over their representation categories. We show that
this bimodule category is invertible if and only if the actions are
free, with finite dimensional fixed point algebras, which are in
duality as Frobenius algebras in an appropriate sense. This
extends the well-known characterization of monoidal equivalence in
terms of bi-Hopf--Galois objects.
\end{abstract}

\date{April 16, 2017; minor corrections May 28, 2018}

\maketitle

\section*{Introduction}

The Tannaka--Krein duality principle, which roughly says that a
quantum group is characterized by its representation category viewed
as a concrete category of vector spaces, has played fundamental role
in the development of various approaches to quantum groups. In
mathematical physics, the attempts by the Leningrad school to find
an algebraic structure behind the solutions of Yang--Baxter
equations ($R$-matrices) led to the famous Drinfeld--Jimbo quantized
universal enveloping algebras, where $R$-matrices are regarded as
intertwiners of tensor products of representations of Hopf algebras.
In the operator algebraic framework, Woronowicz's Tannaka--Krein
duality theorem has been used to construct many examples of compact
quantum groups beyond the $q$-deformations, see, e.g.,~\cite{MR2718205}.

One natural question arising from this principle is the following:
which categorical concepts for representation categories of quantum
groups admit Hopf algebraic formulations?
For example, the most fundamental question of when one has an
equivalence $\Rep G_1 \cong \Rep G_2$ as abstract monoidal
categories, has a very satisfactory answer due to
Schauenburg~\cite{MR2075600} building on an earlier work of
Ulbrich~\cite{MR915993}. It says that the
representation categories are monoidally equivalent precisely when
there is a $G_1$-$G_2$-Hopf--Galois object,
which is an algebra with commuting coactions of the function
algebras of $G_1$ and $G_2$, which are separately `free' and
`transitive' (or `ergodic').

In the operator algebraic context, the C$^*$-analogue of this
characterization~\cite{MR2202309} has been fruitfully used by many
authors to deduce analytic properties of one compact or discrete
quantum group from another, starting from the work of Vaes and
Vergnioux~\cite{MR2355067}, where they showed that exactness of the
reduced function algebra of a compact quantum group is invariant
under monoidal equivalence. More recently, induction of central
multipliers along bi-Hopf--Galois objects was used to show that free quantum
groups have the
Haagerup property and the weak amenability~\citelist{\cite{MR3084500}\cite{MR3238527}}. Another
interesting development is the introduction of central property (T)
by Arano~\cite{arXiv:1410.6238}, which suggests that there is a
close connection between harmonic analysis on the representation
categories of the $q$-deformations of compact Lie groups and the
classical theory of unitary representations of complex semisimple
Lie groups.

These works have led to a study of analytic properties of
C$^*$-tensor categories, which also has roots in Popa's earlier work
on approximation properties of standard invariants of
subfactors~\cite{MR1729488}.  Indeed, as has been shown by Popa and
Vaes~\cite{MR3406647} and the authors~\cite{MR3509018},  `central'
approximation properties of quantum groups considered in~\citelist{\cite{MR3238527}\cite{arXiv:1410.6238}} can be formulated at the purely categorical level. As
one of the applications, this has allowed one to unify property (T)
for (quantum) groups and property (T) for subfactors.

One crucial insight from the subfactor theory is that there is a
more interesting equivalence relation on tensor categories beyond
the mere equivalence. It corresponds to exchanging a subfactor for
its dual inclusion, and in the case of fusion categories, a relevant
notion was introduced by M\"{u}ger~\cite{MR1966524} under the name
of \emph{weak monoidal Morita equivalence}, which is now also called
\emph{categorical Morita equivalence}. Namely, two fusion categories
$\CC_1$ and $\CC_2$ are called weakly monoidally Morita equivalent
if one of them is monoidally equivalent to the category of bimodules
over a Frobenius algebra in the
other~\citelist{\cite{MR1257245}\cite{MR2075605}}. In a more
symmetric form this can be formulated in terms of $2$-categories, or
as existence of an invertible $\CC_1$-$\CC_2$-module
category~\cite{MR2677836}, see Section~\ref{sec:invertible} for
precise definitions. Yet another characterization is that the
Drinfeld centers of $\CC_1$ and $\CC_2$ are equivalent as braided
monoidal categories~\citelist{\cite{MR1822847}\cite{MR2735754}}. Most of
these admit straightforward generalizations to the setting of infinite C$^*$-tensor categories,
although a characterization of categorical Morita equivalence in terms of the Drinfeld center seems to remain as an interesting problem.

Popa's work on subfactors implies that sensible analytic properties
should be invariant under categorical Morita equivalence. For
central property (T), this is indeed the
case~\citelist{\cite{MR3509018}\cite{arXiv:1511.06332}}. It is therefore natural to expect that categorical Morita equivalence should be useful in
studying analytic properties of quantum groups.

The goal of the present paper is to give an algebraic
characterization of categorical Morita equivalence for
representation categories of compact quantum groups. By the
Tannaka--Krein type duality for quantum group
actions~\citelist{\cite{MR1976459}\cite{MR3121622}\cite{MR3426224}},
bimodule categories over representation categories correspond to
C$^*$-algebras with commuting actions of the quantum groups. Namely,
given commuting actions of compact quantum groups $G_1$ and $G_2$ on
a unital C$^*$-algebra $A$, we can consider the category $\DD_A$ of
equivariant finitely generated right Hilbert $A$-modules. Therefore
the precise question we are going to answer is the following: under
what conditions is the category $\DD_A$ invertible as a $(\Rep
G_2)$-$(\Rep G_1)$-module category?

Our main result (Theorem~\ref{thm:main}) states that $\DD_A$ is
invertible if and only if the actions are separately free, have
finite dimensional fixed point algebras $A^{G_1}$ and $A^{G_2}$, and
that these algebras sit nicely in $A$ so that the equivariant
$A$-modules $A^{G_1} \otimes A$ and $A^{G_2} \otimes A$ are
isomorphic in a way that respects the actions of $A^{G_1}$ and
$A^{G_2}$, which we call the \emph{$G_1$-$G_2$-Morita--Galois condition}.
When the actions are ergodic, so that the fixed point subalgebras
are trivial, we recover the bi-Hopf--Galois condition.

Finally, let us note that $2$-categories have close connection to
the theory of quantum groupoids, and a result of De Commer and
Timmermann gives a characterization of categorical Morita
equivalence of compact quantum groups in terms of what they call
\emph{partial compact quantum groups}~\cite{MR3353033}. Their construction gives a `zigzag' of the so-called linking and co-linking quantum groupoids between $G_1$ and $G_2$, as opposed to our one-step construction. They also do not give any characterization of the `off-diagonal' parts of the co-linking groupoids, and the overall construction involving Hayashi's canonical partial quantum groups associated with the representation categories of $G_1$ and $G_2$ seems to be more involved than ours. At the same time their construction works beyond our setting of compact quantum groups and allows one to capture categorical Morita equivalence of several (partial) compact quantum groups at once. It would be an interesting problem to find a characterization of the `off-diagonal' parts of their co-linking groupoids and compare it with our results, but we do not attempt to go in this direction in the present paper.

\smallskip

The paper consists of four sections and an appendix.
Section~\ref{sec:prelim} is a recollection of basic conventions and
results that we use freely throughout the paper.

Section~\ref{sec:Frobenius} is also of preliminary nature. Here we
discuss Frobenius algebras in C$^*$-tensor categories and modules
over them, and compare such algebras in the category of finite dimensional Hilbert spaces to finite dimensional C$^*$-algebras with prescribed faithful states (Frobenius C$^*$-algebras).
A large part of this material is surely known to experts,
but for the lack of a comprehensive reference we provide proofs of
many results.

Section~\ref{sec:invertible} is the main part of the paper. Here we
formulate and prove our main result indicated above.
We also provide a one-sided variant (Theorem~\ref{thm:one-sided-Morita-Galois}) which starts from a single quantum group $G$ and its action, and then produces another categorically Morita equivalent quantum group, generalizing the notion of $G$-Hopf--Galois objects.

In Section~\ref{sec:relten} we discuss relative tensor products of
invertible bimodule categories which correspond to the transitivity of Morita equivalence, as well as give a few examples.

In Appendix we discuss a correspondence between module categories
and Frobenius algebras. In the purely algebraic setting the
existence of such a correspondences was established by
Ostrik~\cite{MR1976459}. Its adaption to the C$^*$-setting is
formulated in~\cite{arXiv:1511.07982}, but we believe certain points
concerning unitarity deserve further explanation.

\smallskip

\paragraph{\bf Acknowledgement} {\small
The results of this paper were obtained during the first author's
visits to the University of Tokyo and Ochanomizu University. The
preparation of the paper was completed during his visit to the Texas
A\&M University. It is his pleasure to thank the staff of these
universities for hospitality, and Yasuyuki Kawahigashi and Ken
Dykema for making these visits possible.}

\section{Preliminaries}\label{sec:prelim}

\subsection{Quantum groups and tensor categories}

For a detailed discussion of C$^*$-tensor categories and compact
quantum groups we refer the reader to~\cite{MR3204665}. Let us just
recall a few basic definitions and facts.

\smallskip

A \emph{C$^*$-category} is a category $\CC$ where the morphisms sets
$\CC(U,V)$ are complex Banach spaces endowed with complex conjugate
involution $\CC(U, V) \to \CC(V, U)$, $T \mapsto T^*$, satisfying
the C$^*$-identity (so that every endomorphism ring
$\CC(X)=\CC(X,X)$ becomes a C$^*$-algebra) and having the property
$T^*T\ge0$ in $\cC(X)$ for any $T \in \cC(X, Y)$. The most basic
example of such a category is $\Hilb_f$, the category of finite
dimensional Hilbert spaces. We tacitly assume that $\CC$ is closed
under finite direct sums and subobjects, which means that any
idempotent in the endomorphism ring~$\CC(U)$ comes from a direct
summand of~$U$.

A \emph{unitary functor}, or a \emph{C$^*$-functor}, $F\colon \cC
\to \cC'$ between C$^*$-categories is a $\C$-linear functor from
$\cC$ to $\cC'$ satisfying $F(T^*) = F(T)^*$.

A \emph{C$^*$-tensor category} is a C$^*$-category
endowed with a unitary bifunctor $\otimes \colon
\CC \times \CC \to \CC$, a distinguished object $\un \in \CC$, and
natural unitary isomorphisms
\begin{align*} \un \otimes U &\cong U \cong U \otimes \un,& \Phi&\colon (U \otimes V) \otimes W \to U \otimes (V \otimes W)
\end{align*} satisfying certain compatibility conditions. Unless said
otherwise, we always assume that $\CC$ is \emph{strict}, that is,
the above isomorphisms are the identity morphisms, but thanks to a
C$^*$-analogue of Mac Lane's coherence theorem this does not lead to
loss of generality. We also assume that the unit $\un$ is simple.

A \emph{C$^*$-$2$-category} on a set $I$ of `$0$-cells' is given by
a collection of C$^*$-categories $\cC_{st}$ for $s,t \in I$
together with bilinear unitary bifunctors $\cC_{rs} \times
\cC_{st}\to\cC_{rt}$ and unit objects $\un_s\in\cC_{ss}$. The axioms
which this structure satisfies are analogous to those of strict
C$^*$-tensor categories. In other words, the main difference from
the latter categories is that the tensor product $X \otimes Y$ is
defined not for all objects, but only when $X\in \cC_{rs}$ and $Y\in
\cC_{st}$, and then $X \otimes Y\in\cC_{r t}$. Again, it is possible
to consider a non-strict version, \emph{C$^*$-bicategories}, but we
do not do this as there is no essential loss of generality in considering only C$^*$-$2$-categories.

A \emph{unitary tensor functor}, a \emph{unitary monoidal functor},
or a \emph{C$^*$-tensor functor}, $\CC\to\CC'$ between C$^*$-tensor
categories is a pair consisting of a unitary functor $F\colon
\CC\to\CC'$, such that $F(\un_\CC)\cong\un_{\CC'}$, and a collection
$F_2$ of natural unitary isomorphisms $F(U)\otimes F(V)\to
F(U\otimes V)$ such that $F_2(F_2\otimes\iota)=F_2(\iota\otimes
F_2)\colon F(U)\otimes F(V)\otimes F(W)\to F(U\otimes V\otimes W)$.
If $F(\un_\CC)=\un_{\CC'}$, $F(U\otimes V)= F(U)\otimes F(V)$ and
$F_2\colon F(U)\otimes F(V)\to F(U)\otimes F(V)$ are the identity
morphisms, then we say that we have a strict tensor functor.

A C$^*$-tensor category $\CC$ is called \emph{rigid}, if every
object $U$ has a conjugate object, that is, there exist an object
$\bar U$ and morphisms $R\colon\un\to \bar U\otimes U$ and $\bar
R\colon\un\to U\otimes\bar U$ solving the conjugate equations
\begin{align*}
(R^*\otimes\iota_{\bar U})(\iota_{\bar U}\otimes\bar R)&=\iota_{\bar
U},&
(\bar R^*\otimes\iota_U)(\iota_U\otimes R)&=\iota_U.
\end{align*}
The minimum $d(U)$ of the numbers $\|R\|\,\|\bar R\|$ over all
solutions is called the \emph{intrinsic dimension} of $U$. A solution $(R,\bar R)$ is
called \emph{standard} if $\|R\|=\|\bar R\|=d(U)^{1/2}$. Any
standard solution $(R_U,\bar R_U)$ defines a trace on $\CC(U)$ by
$$
\Tr_U(T)=R_U^*(\iota\otimes T)R_U,
$$
which is independent of any choices and is also equal to $\bar R_U^*(T\otimes \iota)\bar
R_U$ ({\em sphericity}). The normalized trace
$d(U)^{-1}\Tr_U$ is denoted by $\tr_U$. More generally, we have
\emph{partial categorical traces}
$$
\Tr_U\otimes\iota\colon\CC(U\otimes V,U\otimes W)\to\CC(V,W),\ \
T\mapsto(R^*_U\otimes\iota)(\iota\otimes T)(R_U\otimes\iota),
$$
and similarly $\iota\otimes\Tr_U\colon\CC(V\otimes U,W\otimes
U)\to\CC(V,W)$. Once standard solutions are fixed, we can define a $*$-preserving
anti-multiplicative map $\cC(U,V)\to\cC(\bar V,\bar U)$, $T\mapsto
T^\vee$, such that
\begin{align*}
(\iota\otimes T)R_U&=(T^\vee\otimes\iota)R_V,&
(T\otimes\iota)\bar R_U&=(\iota\otimes T^\vee)\bar R_V.
\end{align*}

Rigidity can be similarly formulated for C$^*$-$2$-categories.
Briefly, in the above notation a dual of $X \in \cC_{s t}$ is given
by an object $\bar{X} \in \cC_{t s}$ and morphisms $R \colon \un_t
\to \bar{X} \otimes X$, $\bar{R}\colon \un_s \to X \otimes \bar{X}$
satisfying the conjugate equations of the same form. The dimension
$d(X)$ and standard solutions $(R_X, \bar{R}_X)$ make sense, and the
functional $\Tr_X(T) = R_X^*(\iota \otimes T) R_X^*$ on
$\cC_{st}(X)$ is tracial and satisfies the sphericity condition.

\smallskip

An example of a rigid C$^*$-tensor category is the representation
category of a \emph{compact quantum group}. A compact quantum group
$G$ is represented by a unital C$^*$-algebra $C(G)$ equipped with a
unital $*$-homomor\-phism $\Delta \colon C(G)\to C(G)\otimes C(G)$
satisfying the coassociativity $(\Delta \otimes \iota) \Delta =
(\iota \otimes \Delta) \Delta$ and the cancellation property,
meaning that $(C(G) \otimes 1) \Delta(C(G))$ and $(1 \otimes C(G))
\Delta(C(G))$ are dense in $C(G) \otimes C(G)$.  There is a unique
state $h$ satisfying $(h \otimes \iota) \Delta = h(\cdot)1$ and
$(\iota \otimes h) \Delta = h(\cdot)1$ called the \emph{Haar state}.
If $h$ is faithful, then~$G$ is called a \emph{reduced} quantum
group. Throughout the whole paper we only consider such quantum
groups.

A finite dimensional \emph{unitary representation} of $G$ is a
unitary element $U\in B(H_U)\otimes C(G)$, where~$H_U$ is a finite
dimensional Hilbert space, such that
$(\iota\otimes\Delta)(U)=U_{12}U_{13}$. The tensor product of two
representations~$U$ and~$V$ is defined by $U\circt V=U_{13}V_{23}$.
This turns the category  $\Rep G$ of finite dimensional unitary
representations of $G$ into a rigid C$^*$-tensor category.

The duality in the category $\Rep G$ can be described as follows.
Take the \emph{regular algebra} $\C[G]$ of $G$, which is the dense
$*$-subalgebra of $C(G)$ spanned by the matrix coefficients of
finite dimensional representations. It is a Hopf $*$-algebra, with
the antipode characterized by $(\iota \otimes S)(U) = U^*$ for any
unitary representation~$U$. Consider the dual space $\U(G)=\C[G]^*$.
It has the structure of a $*$-algebra, defined by duality from the
Hopf $*$-algebra structure on~$\C[G]$. Every finite dimensional
unitary representation $U$ of $G$ defines a
$*$-representation~$\pi_U$ of~$\U(G)$ on~$H_U$ by
$\pi_U(\omega)=(\iota\otimes\omega)(U)$. We often omit $\pi_U$ in
expressions and write~$\omega\xi$ instead of $\pi_U(\omega)\xi$.
There is a canonical positive element $\rho\in\U(G)$, called the
\emph{Woronowicz character}, characterized by
\begin{align*}
(\iota\otimes S^2)(U)
&=(\pi_U(\rho)\otimes1)U(\pi_U(\rho^{-1})\otimes1),& \Tr\pi_U(\rho)
&=\Tr\pi_{U}(\rho^{-1})
\end{align*}
for any finite dimensional unitary representation $U$.

The (non-unitary) \emph{contragredient} representation of $U$ is
given by $U^c=(j\otimes\iota)(U^*)\in B(\bar H_U)\otimes\C[G]$,
where~$j$ denotes the canonical $*$-anti-isomorphism $B(H_U)\cong
B(\bar H_U)$ defined by $j(T)\bar\xi=\overline{T^*\xi}$. Its
unitarization, the \emph{conjugate} unitary representation
$\bar{U}$, is given by
$$
\bar U=(j(\pi_U(\rho))^{1/2}\otimes1) U^c (j(\pi_U(\rho))^{-1/2}\otimes1).
$$
Although $S$ does not satisfy $S^2 = \iota$ nor $S(x^*) = S(x)^*$ (which are in fact equivalent) and is not bounded
on $C(G)$ in general, the \emph{unitary antipode} $R$, which is characterized by $(j
\otimes R)(U) = \bar{U}$, does satisfy these properties.

Finally, standard solutions of the conjugate equations can be defined by
\begin{equation}
\label{eq:std-sol-from-Woronowicz-char}
R_U(1)=\sum_i\bar\xi_i\otimes\rho^{-1/2}\xi_i\ \ \text{and}\ \ \bar
R_U(1)=\sum_i\rho^{1/2}\xi_i\otimes\bar\xi_i,
\end{equation}
where $(\xi_i)_i$ is an orthonormal basis in $H_U$. Note that for
this choice of standard solutions we have $T^\vee=j(T)$. The above
expressions for standard solutions imply that the dimension $d(U)$
coincides with the quantum dimension $\dim_q
U=\Tr\pi_U(\rho)=\Tr\pi_U(\rho^{-1})$. We also have $j(\pi_U(\rho)) =
\pi_{\bar{U}}(\rho^{-1})$ on $\bar{H}_U = H_{\bar{U}}$.

\smallskip

For a unitary representation $U$, it will often be convenient to
view $H_U$ either as a unitary right comodule over $\C[G]$ by
letting $\delta_U(\xi)=U(\xi\otimes1)$, or as a unitary left
comodule by letting $\delta_U(\xi)=U^*_{21}(1\otimes\xi)$. This
should not cause any confusion, as for a fixed compact quantum group
we always use only one point of view depending on whether we
consider right or left comodule algebras, and that will always be
clearly stated. Note that if we consider the spaces $H_U$ as right
comodules, the tensor product of representations of $G$ corresponds
to the tensor product of right comodules, while if we consider $H_U$
as left comodules, it corresponds to the opposite tensor product of
left comodules.

\subsection{Tannaka--Krein duality for quantum group actions}
\label{sec:TK}

A \emph{left action} of a compact quantum group~$G$ on a unital
C$^*$-algebra $A$ is represented by an injective unital $*$-homomorphism
$\alpha\colon A\to C(G)\otimes A$ such that
$(\Delta\otimes\iota)\alpha=(\iota\otimes\alpha)\alpha$, and that
$(C(G)\otimes1)\alpha(A)$ is dense in $C(G)\otimes A$. In this case
we say that $A$ is a left $G$-C$^*$-algebra, and also write $G \curvearrowright A$ to express this situation. Given such an
algebra, we have a distinguished subalgebra $\A\subset A$, called the \emph{regular subalgebra}, spanned by the
elements $a$ (the \emph{regular elements}) such that $\alpha(a)$ lies in the algebraic tensor
product of $\C[G]$ and $A$. More concretely, $\A$ is the linear
span of elements $(h(\cdot\, x)\otimes\iota)\alpha(a)$, where
$x\in\C[G]$ and $a\in A$. It is a dense unital $*$-subalgebra of $A$, and the restriction of $\alpha$ to $\A$ turns it into a
left $\C[G]$-comodule algebra in the purely algebraic sense.

Consider the fixed point C$^*$-algebra
$$
B = A^G = \{a\in A\mid\alpha(a)=1\otimes a\}.
$$
Denote by $\Corr(B)$ the C$^*$-tensor
category of C$^*$-correspondences over $B$, that is, the category of
right Hilbert $B$-modules $X$ equipped with a unital
$*$-homomorphism from $B$ into the C$^*$-algebra of adjointable maps
on $X$. This category is not rigid and generally it has nonsimple
unit. We will mostly be interested in the case when $B$ is finite
dimensional, and instead of $\Corr(B)$ we will work with its full
subcategory $\bimodcat{B}$ of finite dimensional correspondences.

Define a functor $F\colon\Rep G\to\Corr(B)$ by
$$
F(U)=(H_U\otimes A)^{G}=(H_U\otimes\A)^G.
$$
Here, according to our convention, we view $H_U$ as a left
$\C[G]$-comodule, since $\A$ is a left comodule algebra, and then
$F(U)$ is the space of invariant vectors in the tensor product of
comodules $H_U$ and $\A$. Note that if we did consider $H_U$ as a
right comodule, then we could write $F(U)$ as the cotensor product
$H_U\boxvoid_G \A$. The $B$-valued inner product on $(H_U\otimes
\A)^G$ is obtained by restricting the obvious $A$-valued inner
product on $H_U\otimes A$: $\langle \xi\otimes a,\zeta\otimes
b\rangle_A=(\zeta,\xi)a^*b$.\footnote{Our convention is that inner products on Hilbert spaces are linear in the first variables, while those on right Hilbert modules over C$^*$-algebras are linear in the second variables.} We then have natural isometries
\begin{equation}
\label{eq:spec-ftr-weak-tens-str}
F_2\colon F(U)\otimes_B F(V)\to F(U\circt V),\quad x \otimes y\mapsto
x_{13}y_{23}.
\end{equation}
The pair $(F,F_2)$ is called the \emph{spectral functor} defined by
the action $\alpha$. In general the isometries $F_2$ are not
unitary, so it is not a tensor functor but only a \emph{weak}, or
\emph{lax}, tensor functor.

Properties of spectral functors can be axiomatized and this way we
get a one-to-one correspondence between the isomorphism classes of
unital left $G$-C$^*$-algebras and natural unitary monoidal
isomorphism classes of weak unitary tensor
functors~\citelist{\cite{MR2358289}\cite{MR3426224}}. We will only
need to know how an action $\alpha\colon A\to C(G)\otimes A$ can be
reconstructed from the corresponding spectral functor $(F,F_2)$.

Consider the set $\Irr(G)$ of equivalence classes of irreducible
representations of $G$ and choose representatives $U_i\in
B(H_i)\otimes C(G)$ for $i\in\Irr(G)$. As a $G$-space, $\A$ can be
identified with
$$
\bigoplus_{i \in \Irr(G)} \bar H_i\otimes F(U_i),
$$
endowed with a left action of $G$ given by
$$
\alpha(\bar\xi\otimes x)=(U^c_i)^*_{21}(1\otimes\bar\xi\otimes x).
$$
For $\bar\xi\otimes x \in \bar H_i\otimes F(U_i)$ and
$\bar\zeta\otimes y\in \bar H_j\otimes F(U_j)$, their product is given
by
$$
(\bar\xi\otimes x)(\bar\zeta\otimes y)=\sum_k
\overline{w^*_k(\xi\otimes\zeta)}\otimes F(w^*_k)F_2(x \otimes y),
$$
where $w_k\in\Mor(U_{l_k},U_i\circt U_j)$ are isometries defining a
decomposition of $U_i\circt U_j$ into irreducible representations.
The involution is given by
$$
(\bar\xi\otimes x)^*=\overline{(\iota\otimes\xi^*)R_i(1)}\otimes
S^*_x F(\bar R_i)(1),
$$
where $R_i\colon\un\to U_{\bar i}\circt U_i$ and $\bar
R_i\colon\un\to U_i\circt U_{\bar i}$ form a solution of the
conjugate equations and $S_x\colon F(U_{\bar i})\to F(U_i\circt
U_{\bar i})$ is the map $y\mapsto F_2(x \otimes y)$.

\smallskip

A related categorical characterization of actions of $G$ is in terms
of module categories. Recall first that, given a C$^*$-tensor
category $\CC$, a \emph{right $\CC$-module category} is a
C$^*$-category $\DD$ together with a unitary bifunctor
$\otimes\colon \DD\times\CC\to\DD$ and natural unitary isomorphisms
$X\otimes \un\cong X$ and $X\otimes (U\otimes V)\cong (X\otimes
U)\otimes V$ satisfying certain compatibility conditions. For a
strict C$^*$-tensor category $\CC$, a module category is called
strict if these isomorphisms are just the identity morphisms, and
unless explicitly stated otherwise we will assume that we deal with
such module categories, which again does not lead to loss of generality. An equivalent way to define a right
$\CC$-module structure on a C$^*$-category~$\DD$ is by saying that
we have a unitary tensor functor from $\CC^{\otimes\op}$ into the
C$^*$-tensor category $\End(\DD)$ of unitary endofunctors of $\DD$,
which has uniformly bounded natural transformations as its morphisms.

A \emph{unitary $\CC$-module functor} between right $\CC$-module categories
$\DD$ and $\DD'$ is given by a pair $(F, \theta)$ consisting of a unitary functor
$F\colon\DD\to\DD'$ and a collection $\theta$ of natural unitary isomorphisms
$\theta_{X,U}\colon F(X)\otimes U\to F(X\otimes U)$
satisfying some compatibility conditions, which in the case of
strict module categories become $\theta_{X\otimes
U,V}(\theta_{X,U}\otimes\iota)=\theta_{X,U\otimes V}\colon
F(X)\otimes U\otimes V\to F(X\otimes U\otimes V)$. We denote by
$\End_\CC(\DD)$ the C$^*$-tensor category of unitary
$\CC$-module endofunctors of $\DD$.

Returning to an action $\alpha\colon A\to C(G)\otimes A$, consider
the category $\DD_A$ of $G$-equivariant finitely generated right
Hilbert $A$-modules. We will sometimes denote this category more
suggestively by $\rmodcatin{A}{G}$. Thus, the objects of $\DD_A$ are
right Hilbert $A$-modules $X$ equipped with isometries
$\delta_X\colon X\to C(G)\otimes X$, where we consider $C(G)\otimes
X$ as a right Hilbert $(C(G)\otimes A)$-module, satisfying the
following properties: $\delta_X(X)(C(G)\otimes1)$ is dense in
$C(G)\otimes X$,
$(\Delta\otimes\iota)\delta_X=(\iota\otimes\delta_X)\delta_X$,
$\delta_X(\xi a)=\delta_X(\xi)\alpha(a)$, and
$\langle\delta_X(x),\delta_X(y)\rangle_{C(G) \otimes
A}=\alpha(\langle x,y\rangle_A)$. For any such module
$(X,\delta_X)$, we denote by ${\mathcal X}\subset X$ the dense
subspace spanned by the vectors $x$ such that~$\delta_X(x)$ lies in
the algebraic tensor product of $\C[G]$ and $X$, or more concretely,
the subspace spanned by the vectors $(h(\cdot\,
a)\otimes\iota)\delta_X(x)$, where $a\in\C[G]$ and $x\in X$. Then
$\mathcal X$ is a left $\C[G]$-comodule and a right $\A$-module.

The category $\DD_A$ is a right $(\Rep G)$-module category, the
effect of the action of $U\in\Rep G$ on $X\in\DD_A$ is the
equivariant Hilbert module $H_U\otimes X$. Note once again that
since we consider a left action of $G$ on $A$, we view $H_U$ as a
left comodule. This indeed gives us a right action of $\Rep G$,
since the tensor product of left comodules $H_U$ corresponds to the
opposite tensor product in $\Rep G$. The category $\DD_A$ has a
distinguished object represented by the algebra $A$ itself.

This way we get a one-to-one correspondence between the isomorphism
classes of left $G$-C$^*$-algebras and the unitary isomorphism
classes of pairs $(\DD,X)$, where $\DD$ is a right $(\Rep G)$-module
category and~$X$ is a generating object in $\DD$, meaning that any
other object is a subobject of $X\otimes U$ for some $U\in\Rep
G$~\citelist{\cite{MR3121622}\cite{MR3426224}}. If we chose another
generating object, then we get a $G$-equivariantly Morita equivalent
C$^*$-algebra. Therefore we also get a one-to-one correspondence
between the $G$-equivariant Morita equivalence classes of left
$G$-C$^*$-algebras and the unitary equivalence classes of singly
generated right $(\Rep G)$-module categories. For finite quantum
groups and their actions on finite dimensional algebras, which can
then be considered as algebra objects in $\Rep G$, similar results
were already obtained by Ostrik~\cite{MR1976459}.

The relation between the above two categorical descriptions can be
described as follows. Assume we have a pair $(\DD,X)$ as above. Then
we can consider the unital C$^*$-algebra $B=\DD(X)$ and define a
weak unitary tensor functor $\Rep G\to\Corr(B)$ by letting
$$
F(U)=\DD(X,X\otimes U),
$$
with the $B$-valued inner product $\langle S,T\rangle_B=S^*T$, and the tensor structure
$$
F_2\colon F(U)\otimes_B F(V)\to F(U\circt V),\ \ S\otimes T\mapsto
(S\otimes\iota)T.
$$
Thus, for example, the formula for involution on $\A=\bigoplus_i\bar
H_i\otimes\DD(X,X\otimes U_i)$ becomes
$$
(\bar{\xi} \otimes T)^* = \overline{(\iota\otimes\xi^*)R_i(1)} \otimes (T^* \otimes \iota_{\bar{\imath}}) (\iota_X \otimes \bar{R}_i)
$$
for $\xi \in H_i$ and $T \in \DD(X, X\otimes U_i)$.

\smallskip

Of course, everything above makes sense also for right actions
$\alpha\colon A\to A\otimes C(G)$ and left module categories.
Briefly, given such an action and letting $B=A^G$, the corresponding
spectral functor is defined by
\begin{align*}
F&\colon (\Rep G)^{\otimes\op}\to \Corr(B),\quad F(U) =(H_U\otimes A)^G,\\
F_2&\colon F(U)\otimes_{B}F(V)\to F(V\circt U),\quad X\otimes Y \mapsto
X_{23}Y_{13}.
\end{align*}
The dense subalgebra $\A\subset A$ is reconstructed from $(F,F_2)$ by the same formula
as before, $\A=\bigoplus_i\bar H_i\otimes F(U_i)$, endowed with the right action of $G$ given by
$$
\alpha(\bar\xi\otimes x)=(U_i^c)_{13}(\bar\xi\otimes x\otimes 1).
$$
The product is defined similarly to the case of left actions. The
involution is given by
$$
(\bar\xi\otimes x)^*=\overline{(\xi^*\otimes\iota)\bar
R_i(1)}\otimes S^*_x F(R_i)(1)
$$
for $\bar\xi\otimes x \in \bar{H}_i \otimes F(U_i)$, where
$S_x(y)=F_2(x \otimes y)$.

The left $(\Rep G)$-module category $\DD_A$ is defined in the same
way as before, but equivariant right Hilbert $A$-modules are now
right $C(G)$-comodules. The spectral functor and the pair
$(\DD_A,X)$, where $X=A\in\DD_A$, are related by
$F(U)=\DD_A(X,U\otimes X)$.

\begin{remark}
In \cite{MR3121622}, a right $G$-C$^*$-algebra corresponding to a
left module category $\cD$ and $X \in \cD$ was constructed as the
completion of $\tilde{\A} = \bigoplus_i \cD(U_i\otimes  X, X) \otimes H_i$,
where $H_i$ has the coaction $\xi \mapsto U_i (\xi \otimes 1)$. This
approach is of course equivalent to the one above, with the
isomorphism $\tilde{\A} \to \A$ given by
$$
\cD(\bar{U}_i \otimes X, X) \otimes \bar{H}_i \to \bar{H}_i \otimes \cD(X,
U_i \otimes X), \quad S \otimes \bar{\xi} \mapsto (\iota \otimes \xi^*)
R_i(1) \otimes (\iota_i \otimes S) (\bar{R}_i \otimes \iota_X).
$$
In particular, when the weak tensor functor $F$ is actually the
fiber functor of $G$, $\tilde{A}$ can be identified with $\C[G]$ on
the nose, while the above map gives a right $G$-C$^*$-algebra
isomorphism $\C[G] \cong \A$.
\end{remark}

\begin{remark}\label{rem:left-right}
Our correspondence between left actions and right module categories
instead of left ones is more of a convention than a necessity. Given
any right $(\Rep G)$-module category $\DD$ we can reverse the
directions of arrows in $\DD$ to get a category $\DD^\op$, and then
define a left action of $\Rep G$ on $\DD^\op$ using a contravariant
functor $U\mapsto \bar U$. At the level of C$^*$-algebras this
corresponds to passing from a left action $\alpha\colon A\to
C(G)\otimes A$ to the right action $a\mapsto (\iota\otimes
R)(\alpha(a)_{21})$ on $A^\op$, where $R$ is the unitary antipode
on~$C(G)$. Concretely, the anti-isomorphism of the algebras
corresponding to $(\DD,X)$ and $(\DD^\opos,X^\opos)$ is given by
\begin{equation*}
\begin{split}
\bar{H}_i \otimes \cD(X, X \otimes U_i) &\to \bar{H}_i \otimes \cD^\opos(X^\opos, U_i \otimes X^\opos) = \bar{H}_i \otimes \cD(X \otimes \bar{U}_i, X),\\
\bar{\xi} \otimes S &\mapsto (\iota \otimes \xi^*) R_i(1) \otimes (\iota_X \otimes \bar{R}_i^*)(S \otimes \iota_{\bar{\imath}}).
\end{split}
\end{equation*}
\end{remark}

\subsection{Free actions}\label{sec:free}

A left action $\alpha\colon A\to C(G) \otimes A$ of a compact quantum
group on a unital C$^*$-algebra is called
\emph{free}~\cite{MR1760277}, if $(1 \otimes A) \alpha(A)$ is dense in
$C(G) \otimes A$. By now there are many equivalent characterizations
of freeness~\citelist{\cite{MR3141721}\cite{MR3665403}}. In
particular, freeness is equivalent to any of the following
conditions:

\begin{itemize}
\item the regular subalgebra $\A\subset A$ is a
\emph{Hopf--Galois extension} of $B=A^G$, that is, the \emph{Galois~map}
$$
\Gamma\colon \A \otimes_B \A \to \C[G] \otimes \A, \quad a \otimes b \mapsto \alpha(a)(1\otimes b),
$$
is bijective;
\item for any $U \in \Rep G$, the \emph{localized Galois map}
$$
\Gamma_U\colon A_U \otimes_B A \to \C[G]_U \otimes A, \quad a \otimes b \mapsto \alpha(a) (1 \otimes b),
$$
is a unitary map of right Hilbert $A$-modules, where $\C[G]_U$ is
the span of matrix coefficients of $U$, and $A_U = \{a\in A
\mid\alpha(a)\in \C[G]_U \otimes A\}$ is the spectral subspace of
$A$ corresponding to $U$; here $A_U$ has the structure of a right
Hilbert $B$-module induced by the unique $G$-invariant conditional
expectation $E\colon A \to B$, and $\C[G]_U$ is equipped with the
scalar product using the Haar state;
\item the spectral functor $\Rep G \to \Corr(B)$ is
monoidal, that is, the natural isometries
$$
(H_U\otimes A)^G\otimes_B (H_V\otimes A)^G\to (H_{U \circt V}\otimes A)^G
$$
given by~\eqref{eq:spec-ftr-weak-tens-str} are unitary.
\end{itemize}
Note that in the purely algebraic setting the equivalence of these
conditions was established earlier by Ulbrich in the ergodic
case~\cite{MR915993} and by
Schauenburg~\cite{MR2075600}*{Section~2.5} in general.

Freeness for right actions can be characterized similarly, this time
the Galois map being given by
$$
\A \otimes_B \A \to \A \otimes
\C[G], \quad a \otimes b \mapsto \alpha(a) (b \otimes 1).
$$

Yet another characterization of freeness for left actions is as
follows, which in the purely algebraic setting is due to
Schneider~\cite{MR1098988}. Needless to say, there is also a similar
characterization for right actions.

\begin{proposition} \label{prop:imprim}
An action $\alpha\colon A\to C(G) \otimes A$ of a compact
quantum group on a unital C$^*$-algebra~$A$ is free if and only if, for any $G$-equivariant right Hilbert $A$-module $X$, the map $X^G\otimes_{A^G}A\to X$, $x\otimes a\mapsto x a$, is a unitary isomorphism.
\end{proposition}

\bp We follow the idea of~\cite{MR1098988}, but there are several
simplifications due to the cosemisimplicity of~$\C[G]$. Since the
map in the formulation is isometric, the only question is when it
has dense image.  Let us denote the subspace of regular vectors in
$X$ by $\cX$, and look at the product map $\mu\colon X^G \otimes_B
\cA \to \cX$, where $B=A^G$.

First suppose that the action is free. Using the inverse of the Galois map, we can consider the map
$$
\nu\colon \cX \to \cX \otimes_B \cA, \quad x \mapsto x_{(1)}
\Gamma^{-1}(S^{-1}(x_{(0)}) \otimes 1)_1 \otimes
\Gamma^{-1}(S^{-1}(x_{(0)}) \otimes 1)_2.
$$
A standard computation shows that the image of this map is in $(\cX
\otimes_B \cA)^G$, where $G$ acts only on the first factor of
$\cX\otimes_B\cA$. Since the $G$-isotypic decomposition $\cX =
\bigoplus_{i \in \Irr(G)} X_i$ is compatible with the action of $B$,
we have $(\cX \otimes_B \cA)^G = X^G \otimes_B \cA$. Then, using
that the product map $\cA \otimes_B \cA \to \cA$ equals
$(\varepsilon \otimes \iota)\Gamma$, it is easy to check that $\nu$,
considered as a map $\cX\to X^G\otimes_B\A$, is the inverse
of~$\mu$.

Conversely, assume the map in the formulation is unitary for any
$X$, or equivalently, the map $\mu\colon X^G\otimes_B\A\to \cX$ is
an isomorphism. Take $U\in \Rep G$. Consider the equivariant right
Hilbert $A$-module $X=\C[G]_{\bar U}\otimes A$, where the inner
product on $\C[G]_{\bar U}$ is defined by the Haar state. Then
$$
X^G=\{S(a_{(0)})\otimes a_{(1)}\mid a\in A_U\},
$$
so $X^G\cong A_U$ as a right $B$-module. By assumption, the map
$$
X^G\otimes_{B} \A\to{\mathcal X}=\C[G]_{\bar U}\otimes \A,\ \
(S(a_{(0)})\otimes a_{(1)})\otimes b\mapsto S(a_{(0)})\otimes
a_{(1)}b,
$$
is an isomorphism. But this shows that the map $A_U\otimes_{B} \A\to
\C[G]_U\otimes\A$, $a\otimes b\mapsto a_{(0)}\otimes a_{(1)}b$, is
an isomorphism. Hence the localized Galois map $\Gamma_U$ is an
isomorphism. \ep

If an action is free, then it follows from
\cite{MR3141721}*{Corollary~4.2(2)} that, for any $G$-equivariant
finitely generated right Hilbert $A$-module $X$, the module $X^G$ is
finitely generated over $A^G$. Therefore the `only if' part of the
above proposition implies that the functor $X\mapsto X^G$ defines an
equivalence of the categories $\Mod_G\mhyph A$ and $\Mod\mhyph A^G$.
Without the freeness assumption this is not even well-defined as a
functor into the category of finitely generated modules. However, if
the fixed point algebra is finite dimensional, the functor is
well-defined and we get the following characterization of freeness.

\begin{proposition}
\label{prop:freeness} Let $\alpha\colon A\to C(G) \otimes A$ be an
action of a compact quantum group on a unital C$^*$-algebra~$A$.
Assume that $A^G$ is finite dimensional. Then the action is free if
and only if $Y^G\ne0$ for any nonzero $Y \in\cD_A$.
\end{proposition}

\begin{proof}
The `only if' direction follows from the previous proposition and
does not require finite dimensionality of $B=A^G$.

As for the converse, first, we claim that the finite dimensionality
assumption on $A^G$ implies that $\cD_A$ is semisimple. Indeed,
since any module in $\DD_A$ is a direct summand of $H_U\otimes A$
for some $U\in\Rep G$, it suffices to show that the endomorphism
algebra $(B(H_U)\otimes A)^G$ of $H_U\otimes A\in\DD_A$ is finite
dimensional. But this is true, since any faithful $G$-invariant
state $\varphi$ on $B(H_U)$ defines a conditional expectation
$\varphi\otimes\iota\colon (B(H_U)\otimes A)^G\to A^G$ of finite
probabilistic index.

Let $X$ be the object of $\cD_A$ represented by $A$ itself. Then the
space $\cD_A(X, Y) \cong Y^G$ is finite dimensional. In particular,
$Y^G$ is finitely generated over $B$ for any $Y \in \cD_A$. Assume
that the action is not free. Then by the proof of the previous
proposition, there exists $Y \in \cD_A$ of the form
$Y=\C[G]_U\otimes A$ such that the isometric map $Y^G\otimes_{B}
A\to Y$, $x\otimes a\mapsto x a$, is not surjective. Since $Y^G$ is
finitely generated over $B$, the module $Y^G\otimes_{B}A$ is
finitely generated over $A$, hence this map is a morphism in
$\cD_A$. Its image, the proper submodule $Y^G A\subset Y$, has a
nonzero orthogonal complement~$Z$. Clearly, $Z^G=0$.
\end{proof}

\begin{remark}
An equivalent way of formulating the above proposition is as
follows: if a $(\Rep G)$-module category $\cD$ is semisimple, then
the action of $G$ on the C$^*$-algebra corresponding to a generating
object $X \in \cD$ is free if and only if every simple object of
$\cD$ is a subobject of $X$.
\end{remark}

The following observation is useful for checking freeness in
concrete examples.

\begin{proposition} \label{prop:freesubalg}
Assume $\alpha\colon B\to C(G) \otimes B$ is an action of a compact
quantum group $G$ on a unital C$^*$-algebra $B$, and $A\subset B$ is an
invariant C$^*$-subalgebra containing the unit of $B$ such that the
action of $G$ on $A$ is free. Then the action of $G$ on $B$ is also
free.
\end{proposition}

\bp Since the Galois map $\BB\otimes_{B^G}\BB\to\C[G] \otimes \BB$,
$b\otimes c\mapsto\alpha(b)(1\otimes c)$, is always
injective, we only have to check surjectivity. By the freeness of
the action on $A$ the image of this map contains $\C[G] \otimes 1$,
hence it also contains $\C[G] \otimes \BB$. \ep

\subsection{Conventions}
\label{sec:conv}

We often fix representatives $(U_i)_i$ of isomorphism classes of
simple objects in a rigid C$^*$-tensor category, and then use the
subscript $i$ instead of $U_i$, so that we write $d_i$, $(R_i,
\bar{R}_i)$ instead of $d(U_i)$, $(R_{U_i}, \bar{R}_{U_i})$, etc.

In order to simplify various expressions, we often omit the symbols
$\otimes$ and $\circt$ for tensor products of objects in tensor
categories and module categories, as opposed to this preliminary
section. We still write $\otimes$ for tensor products of morphisms
and vector spaces.

When $X$ and $Y$ are objects in a rigid C$^*$-tensor category (or in
a rigid C$^*$-$2$-category) and standard solutions for the corresponding
conjugate equations are fixed, we take $\left((\iota_{\bar{Y}}
\otimes R_X \otimes \iota_Y) R_X, (\iota_X \otimes \bar{R}_Y \otimes
\iota_{\bar{X}}) \bar{R}_X \right)$ as a standard solution for $X
Y(=X\otimes Y)$. We also normalize the choice of standard solutions
in $\Rep G$ as in~\eqref{eq:std-sol-from-Woronowicz-char}. Thus, for
$\xi \in H_i$, we have
\begin{align*}
(\xi^*\otimes\iota)\bar R_i(1) &= \overline{\rho^{1/2} \xi}, &
(\iota \otimes \xi^*) R_i(1) &= \overline{\rho^{-1/2} \xi}.
\end{align*}

Recall once again that all compact quantum groups in this paper are
assumed to be in the reduced form.

When we write formulas for linear maps on subspaces of vector
spaces, such as $(H_U\otimes A)^G\subset H_U\otimes A$, we often
consider only elementary tensors. By this we do not mean that the
subspaces are spanned by such tensors, but that our formulas have
obvious extensions to all the required tensors.

We use the Einstein summation convention, that is, if an index
occurs once in an upper and once in a lower position in an
expression, then we have a sum over this index.

\section{Frobenius algebras and categories of modules}\label{sec:Frobenius}

In this section we collect a number of results on algebra objects in
C$^*$-tensor categories and the corresponding categories of modules.

\subsection{Frobenius algebras in tensor categories}

Recall that a \emph{Frobenius algebra} over $\C$ is a finite dimensional algebra $A$ together with a linear functional $\varphi$ such that the pairing $A\times A\to\C$, $(a,b)\mapsto\varphi(ab)$ is nondegenerate. There are a number of other conditions equivalent to nondegeneracy of the pairing, see, e.g.,~\cite{MR2037238}. One of them is that the vector space $A$ admits a (necessarily unique) coalgebra structure with counit $\varphi$ and coproduct $\Delta\colon A\to A\otimes A$ such that $\Delta$ is an $A$-bimodule map. Explicitly, the coproduct is defined by
\begin{equation} \label{eq:coprod}
\Delta(y)= y x^i\otimes x_i \Bigl(=\sum_i y x^i \otimes x_i\Bigr),
\end{equation}
where $(x_i)_i$ is a basis in $A$ and $(x^i)_i$ is the dual basis,
so that $\varphi(x_ix^j)=\delta_{ij}$.

By a \emph{Frobenius C$^*$-algebra} we mean a finite dimensional C$^*$-algebra $A$ together with a faithful positive linear functional~$\varphi$. Define a scalar product on $A$ by $(x,y)=\varphi(y^*x)$. Then the coproduct $\Delta$ defined by~\eqref{eq:coprod} coincides with the adjoint $m^*$ of the product map $m\colon A\otimes A\to A$, while $\varphi$ equals the adjoint of the map $v\colon\C\to A$, $v(1)=1$. This justifies the following definition.

\begin{definition}[cf.~\citelist{\cite{MR1966524}\cite{MR3308880}}]
An algebra object $(A,m,v)$, with product $m\colon A\otimes A\to A$ and unit $v\colon\un\to A$, in a C$^*$-tensor category $\CC$ is called a \emph{C$^*$-Frobenius algebra} if $m^*\colon A\to A\otimes A$ is an $A$-bimodule morphism, that~is,
\begin{equation*} %\label{eq:Frobenius}
(m\otimes\iota)(\iota\otimes m^*)=m^*m=(\iota\otimes m)(m^*\otimes\iota).
\end{equation*}
\end{definition}

Since the unit $v$ is uniquely determined, we will often write an algebra in $\CC$ as a pair $(A,m)$.

\smallskip

In a similar way, given a C$^*$-Frobenius algebra $(A,m)$, we say that a left $A$-module $(X,m_X\colon A\otimes X\to X)$ in $\CC$ is \emph{unitary} if $m_X^*\colon X\to A\otimes X$ is an $A$-module morphism:
\begin{equation} \label{eq:unitarity}
(m\otimes\iota)(\iota\otimes m_X^*)=m_X^*m_X.
\end{equation}

By the above discussion any Frobenius C$^*$-algebra is a C$^*$-Frobenius algebra in $\Hilb_f$. It is known that the converse is also true. More precisely, we have the following.

\begin{lemma}\label{lem:FrobeniusVScstar}
Let $(A,m,v)$ be a C$^*$-Frobenius algebra in $\Hilb_f$. Then the algebra $A$ admits a unique involution such that it becomes a C$^*$-algebra and such that for the linear functional $\varphi=v^*$ on it we have $(x,y)=\varphi(y^*x)$ for all $x,y\in A$. Also, a left $A$-module $X$ in $\Hilb_f$ is unitary if and only if the representation of $A$ on the Hilbert space $X$ is $*$-preserving.
\end{lemma}

\bp Since $m^*$ is a coproduct with counit $\varphi$, the pair
$(A,\varphi)$ is a Frobenius algebra, so the pairing defined by
$\varphi$ is nondegenerate. Hence we can define an anti-linear
operation $a\mapsto a^*$ on $A$ such that $(x,y)=\varphi(y^*x)$ for
all $x,y\in A$. For $a,b,c\in A$ we have
$$
(c,ab)=(m^*(c),a\otimes b)=((a^*\otimes1)m^*(c),1\otimes b)
=(m^*(a^*c),1\otimes b)=(a^*c,b).
$$
This shows that the left regular representation of $A$ is $*$-preserving, so the $*$-operation is an involution and~$A$ is a C$^*$-algebra.

\smallskip

Next, consider a left $A$-module $X\in\Hilb_f$. By definition, the
representation of $A$ on $X$ is $*$-preserving if, for all $a\in A$
and $x,y\in X$, we have $(ax,y)=(x,a^*y)$. The right hand side can
be written as
$$
(x,m_X(a^*\otimes y))= \left( (m\otimes\iota)(a\otimes m_X^*(x)),1\otimes y
\right)= \left( (v^*m\otimes\iota)(\iota\otimes m_X^*)(a\otimes x), y \right),
$$
so the representation is $*$-preserving if and only if
\begin{equation} \label{eq:unitarity2}
m_X=(v^*m\otimes\iota)(\iota\otimes m_X^*).
\end{equation}
This condition is equivalent to \eqref{eq:unitarity} in any C$^*$-tensor category. Indeed, identity \eqref{eq:unitarity2} follows from \eqref{eq:unitarity} by multiplying the latter by $v^*\otimes\iota$ on the left. Conversely, starting with \eqref{eq:unitarity2} and using that $m=(v^*m\otimes\iota)(\iota\otimes m^*)$ by the Frobenius condition, we compute:
\begin{equation*}
\begin{split}
(m\otimes\iota)(\iota\otimes m_X^*)&=(v^*m\otimes\iota\otimes\iota)(\iota\otimes m^*\otimes\iota)(\iota\otimes m_X^*)\\
&=(v^*m\otimes\iota\otimes\iota)(\iota\otimes\iota\otimes m_X^*)(\iota\otimes m_X^*)\\
&= m_X^*(v^*m\otimes\iota)(\iota\otimes m_X^*)=m_X^*m_X.
\end{split}
\end{equation*}
This completes the proof of the lemma.
\ep

For a C$^*$-Frobenius algebra $A$ in a C$^*$-tensor category $\CC$,
we denote by $\lmodcatin{A}{\CC}$, or simply by $\lmodcat{A}$, the
category of left unitary $A$-modules in $\CC$. It is easy to check
that $\lmodcat{A}$ is a C$^*$-category~\cite{MR3509018}*{p.~418}
using condition \eqref{eq:unitarity2} and arguments similar to the
proof of the above proposition, where we in effect checked that the
fact $m_X$ is a morphism in $\lmodcat{A}$ implies that $m_X^*$ is a
morphism in $\lmodcat{A}$ as well. In a similar way we can introduce
C$^*$-categories $\rmodcatin{A}{\cC}$ of unitary right $A$-modules
and $\bimodcatin{A}{\CC}$ of unitary $A$-bimodules in~$\CC$.

\smallskip

For abstract C$^*$-tensor categories it is natural to consider unitary isomorphisms of C$^*$-Frobenius algebras. But for $\Hilb_f$ there is a larger natural class of isomorphisms.

\begin{lemma}
Let $(A,\varphi_A)$ and $(B,\varphi_B)$ be Frobenius C$^*$-algebras. Assume that $T\colon A\to B$ is an isomorphism of algebras. Consider the adjoint map $T^*\colon B\to A$ with respect to the scalar products defined by $\varphi_A$ and $\varphi_B$. Then $T$ is $*$-preserving if and only if $T^*T\colon A\to A$ is a left $A$-module map.
\end{lemma}

\bp The map $T^*T$ is a left $A$-module map if and only if
\begin{equation}\label{eq:leftmod}
(T^*T(ab),c)=(a T^*T(b),c)
\end{equation}
for all $a,b,c\in A$. The left hand side of \eqref{eq:leftmod} equals
$(T(a)T(b),T(c))$, while the right hand side equals
$$
(T^*T(b),a^*c)=(T(b),T(a^*)T(c)).
$$
We thus see that \eqref{eq:leftmod} holds for all $b,c\in A$ if and only if $T(a)^*=T(a^*)$.
\ep

Motivated by this we give the following definition.

\begin{definition}\label{def:Frob-iso}
We say that an invertible morphism $T\colon A\to A'$ is an \emph{isomorphism of C$^*$-Frobenius algebras} $(A,m)$ and $(A',m')$ in a C$^*$-tensor category $\CC$ if
\begin{equation*}%\label{eq:eqrel}
m'(T\otimes T)=Tm\ \ \text{and}\ \ m(\iota\otimes T^*T)=T^*Tm.
\end{equation*}
\end{definition}

It is straightforward to check that compositions and inverses of
isomorphisms are again isomorphisms. Furthermore, if $(A,m)$ is a
C$^*$-Frobenius algebra and $T\colon A\to A'$ is any invertible
morphism satisfying $m(\iota\otimes T^*T)=T^*Tm$, then by letting
$m'=Tm(T^{-1}\otimes T^{-1})$ we get a C$^*$-Frobenius algebra
$(A',m')$.

\begin{remark}
Instead of requiring $T^*T$ to be a left $A$-module morphism in the above definition we could require $T^*T$ to be a right $A$-module morphism. This would change the notion of an isomorphism, but the isomorphism classes of C$^*$-Frobenius algebras would remain the same. Indeed, assume $T\colon A\to A'$ is an isomorphism according to Definition~\ref{def:Frob-iso}. Consider the polar decomposition $T=u|T|$, so that $|T|$ is a left $A$-module morphism. We have a linear isomorphism
$$
\End_{\Amod}(A)\cong\End_{\modA}(A),\ \ S\mapsto\pi(S)=m(S v\otimes\iota),
$$
which can be characterized by the identity
$$
m(S\otimes\iota)=m(\iota\otimes\pi(S)),
$$
since
$$
m(\iota\otimes\pi(S))=m(\iota\otimes m)(\iota\otimes S v\otimes\iota)=m(m\otimes \iota)(\iota\otimes S v\otimes\iota)
=m(Sm\otimes \iota)(\iota\otimes v\otimes\iota)=m(S\otimes\iota).
$$
Therefore if we let $\tilde T=u\pi(|T|)$, then $\tilde T^*\tilde T\in\pi(|T|)^*\pi(|T|)\in\End_{\modA}(A)$ and
$$
m'=Tm(T^{-1}\otimes T^{-1})=um(|T|^{-1}u^*\otimes u^*)=um(u^*\otimes \pi(|T|)^{-1}u^*)=\tilde Tm(\tilde T^{-1}\otimes \tilde T^{-1}),
$$
proving our claim.
\end{remark}

There are several important subclasses of C$^*$-Frobenius algebras, see again~\cite{MR3308880}.

\begin{definition}
A C$^*$-Frobenius algebra $(A,m,v)$ in $\CC$ is called
\begin{itemize}
\setlength{\itemsep}{3pt}
\item \emph{irreducible}, if $A$ is simple as a left, equivalently, as a right, $A$-module;
\item \emph{simple}, if $A$ is simple as an $A$-bimodule;
\item \emph{special}, or a $Q$-\emph{system}~\cite{MR1257245}, if $mm^*$ is scalar;
\item \emph{standard}, if the pair $(m^*v,m^*v)$ is a standard solution of the conjugate equations for $A$, that is, if $\|m^*v\|^2$ equals the intrinsic dimension $d(A)$ of $A\in\CC$.
\end{itemize}
\end{definition}

\begin{remark}\mbox{\ }

\noindent
(i) We always have a linear isomorphism $\Mor_\CC(\un,A)\cong\End_{\Amod}(A)$, $T\mapsto m(\iota\otimes T)$, with the inverse $S\mapsto S v$. Therefore irreducibility is equivalent to
the condition $\dim\Mor_\CC(\un,A)=1$.

\smallskip\noindent
(ii) As $mm^*$ is an $A$-bimodule morphism, a simple C$^*$-Frobenius algebra is automatically a $Q$-system. In particular, this is true for irreducible C$^*$-Frobenius algebras. Irreducible $Q$-systems are also called ergodic in~\cite{arXiv:1511.07982}.

\smallskip\noindent
(iii) If $(A,m,v)$ is an algebra in $\CC$ such that $mm^*$ is scalar, then it is a $Q$-system, see \cite{MR1444286}*{Section~6} or \cite{MR3308880}*{Lemma~3.5}. Similarly, once we assume that $A$ is a $Q$-system, a left $A$-module $X$ is unitary if and only if $m_X m_X^*$ is scalar, and then it is the same scalar as $mm^*$, see \cite{MR3308880}*{Lemma~3.23} and \cite{MR3509018}*{Lemma~6.1}.

\smallskip\noindent
(iv) Once $mm^*$ is assumed to be scalar, it is natural to fix a normalization of the pair $(m,v)$. For example, we may require this scalar to be $1$. Another natural choice, made  in \cite{MR3509018}, is to require $v$ to be an isometry.

\smallskip\noindent
(v) In~\cite{MR3308880} $Q$-systems are required to be standard, but we do not do this.
\end{remark}

\begin{lemma}\label{lem:nonsimpleFrob}
Any C$^*$-Frobenius algebra is unitarily isomorphic to a direct sum of simple C$^*$-Frobenius algebras.
\end{lemma}

\bp Note that the C$^*$-algebra $\End_{\bimodcat{A}}(A)$ is abelian,
since $A$ is a unit object in the tensor category $\bimodcat{A}$. More
directly, if $S,T\in\End_{\bimodcat{A}}(A)$, then
$$
S T m=Sm(\iota\otimes T)=m(S\otimes T)=Tm(S\otimes\iota)=T S m,
$$
and multiplying on the right by $v\otimes\iota$ we get $ST=TS$.

For every minimal projection $z\in\End_{\bimodcat{A}}(A)$, the
subobject $z A$ of $A$ defined by $z$ becomes a C$^*$-Frobenius
algebra, with product defined by the restriction of $m$ to $z
A\otimes z A$, and $A$ is the direct sum of these algebras. \ep

\subsection{Standard \texorpdfstring{$Q$}{Q}-systems}\label{sec:Q}

Assume $(A,m)$ is a C$^*$-Frobenius algebra. Then $mm^*\colon A\to A$ is an $A$-bimodule morphism, and as was observed in~\cite{MR3308880}*{Lemma~3.5}, this morphism is invertible, so that the product $(mm^*)^{-1/2}m\colon A\otimes A\to A$
defines an isomorphic $Q$-system. We strengthen this observation as follows.

\begin{theorem} \label{thm:standard}
Any C$^*$-Frobenius algebra in a C$^*$-tensor category $\CC$ is isomorphic to a standard $Q$-system.
\end{theorem}

In particular, since isomorphisms of irreducible $Q$-systems are
unitary up to scalar factors, any irreducible $Q$-system is
standard. This has been already observed by M{\"u}ger
in~\cite{MR1966524}*{Remark 5.6(3)}. The general result holds for
similar reasons, but the proof requires a bit more work.

First of all, by Lemma~\ref{lem:nonsimpleFrob} it suffices to prove
the theorem for simple $Q$-systems. Let $(A,m,v)$ be such a
$Q$-system. We may assume that $v$ is an isometry and
$mm^*=\lambda\iota$. Since the object $A$ in $\CC$ is self-dual, by
passing to the subcategory generated by~$A$ we may assume that $\CC$
is rigid. We can then construct a rigid C$^*$-$2$-category $\BB$ of
modules in $\CC$ with the set $\{1,2\}$ of $0$-cells in a standard
way~\citelist{\cite{MR2075605}\cite{MR1966524}\cite{arXiv:1511.06332}}.
Concretely, we take $\BB_{11}=\CC$, $\BB_{22}=\bimodcat{A}$,
$\BB_{12}=\modA$, and $\BB_{21}=\Amod$. The tensor products are
defined over~$A$ when possible, otherwise they are taken in $\CC$.
For a discussion of unitarity of the tensor product $\otimes_A$ and
a proof of (C$^*$-)rigidity of $\BB$, see,
e.g.,~\citelist{\cite{MR3509018}\cite{arXiv:1511.06332}}.

We only want to make two additional remarks. First, the assumption
of standardness made in the above cited papers did not play an
essential role for the construction of $\BB$, the only change is
that $d(A)$ in various formulas has to be replaced by $\lambda$.
Second, it is important to remember that given modules $X\in\modA$
and $Y\in\Amod$, the structure morphism $P_{X,Y}\colon XY\to
X\otimes_A Y$ for the tensor product over $A$ is normalized so that
$P_{X,Y}P_{X,Y}^*=\lambda\iota$, which guarantees the unitarity of
$\otimes_A$.

\bp[Proof of Theorem~\ref{thm:standard}]
Using the above notation, consider $A$ as an object $X$ in $\BB_{12}=\modA$. As a conjugate object $\bar X$ we can take~$A$ considered as an object in $\BB_{21}=\Amod$. We have a solution $(R,\bar R)$ of the conjugate equations for~$X$ defined~by
$$
R=m^*\colon A=\un_2\to \bar X X=A\otimes A,\ \ \bar R=v\colon
\un_\cC=\un_1\to X \bar X= A\otimes_A A=A.
$$
We can find a positive invertible morphism $T\in\End(\bar X)=\End_{\Amod}(A)$ such that the morphisms $R'=(T\otimes\iota)R$ and $\bar R'=(\iota\otimes T^{-1})\bar R$ form a standard solution of the conjugate equations for $X$. Then the formula
$$
R_A=\bar R_A=(\iota\otimes R'\otimes\iota)\bar R'
$$
defines a standard solution of the conjugate equations for $X\bar
X=A$. Note that the morphism $\bar R\colon \un\to A=A\otimes_A A$
lifts to the morphism $\lambda^{-1}m^*v\colon \un\to A\otimes A$,
while $\iota\otimes R\otimes\iota\colon A\otimes_AA\to A\otimes_A
A\otimes A\otimes_A A$ is induced by the morphism $\iota\otimes
m^*v\otimes\iota \colon A\otimes A\to A\otimes A\otimes  A\otimes
A$. Hence we have
\begin{equation*}
\begin{split}
R_A&=(m\otimes m)(\iota\otimes (T\otimes\iota)m^*v\otimes\iota)\lambda^{-1}(\iota\otimes T^{-1})m^*v\\
&=(T\otimes\iota)(m\otimes\iota)(\iota\otimes m^*m)(\iota\otimes v\otimes\iota)\lambda^{-1}m^*T^{-1}v\\
&=(T\otimes\iota)m^*m\lambda^{-1}m^*T^{-1}v =(T\otimes \iota)m^*T^{-1}v.
\end{split}
\end{equation*}
But this means that by letting $m'=m(T\otimes\iota)=T^{-1}m(T\otimes T)$ and $v'=T^{-1}v$ we get an isomorphic C$^*$-Frobenius algebra structure on $A$ with ${m'}^*v'=R_A=\bar R_A$, so this new C$^*$-Frobenius algebra is standard. As it is simple, it is automatically a $Q$-system.
\ep

One advantage of working with standard $Q$-systems is the following result.

\begin{proposition}\label{prop:standard-iso}
Any isomorphism of standard $Q$-systems is unitary up to a scalar factor.
\end{proposition}

\bp
Assume $T\colon (A,m,v)\to (A',m',v')$ is an isomorphism of standard $Q$-systems. Using scalar isomorphisms we may replace these $Q$-systems by isomorphic ones and assume that $m$ and $m'$ are coisometries. Then $\|v\|^2=\|v'\|^2=d(A)$ by standardness. We want to show that $T$ is unitary. By taking the polar decomposition of $T$ and replacing $(A',m',v')$ by a unitarily isomorphic $Q$-system we may further assume that $T$ is a positive morphism, so that in particular $A'=A$ as objects. We then have to show that $T=\iota$.

As $m'(T\otimes T)=Tm$ and $Tm=m(\iota\otimes T)$, we have  $m'(T\otimes\iota)=m$, and then
$$
\Tr_A(T^2)={v'}^*m'(T^2\otimes\iota){m'}^*v'={v'}^*mm^*v'=d(A).
$$
Similarly $\Tr_A(T^{-2})=d(A)$. By the Cauchy--Schwarz inequality we conclude that $T^2$ is the identity morphism, hence $T$ is the identity morphism as well.
\ep

\subsection{Canonical invariant states}

Let us now consider the C$^*$-Frobenius algebras in the representation categories of compact quantum groups.

By a straightforward refinement of Lemma~\ref{lem:FrobeniusVScstar}, the C$^*$-Frobenius algebras in $\Rep G$ correspond to the pairs $(A,\varphi)$ consisting of a finite dimensional right $G$-C$^*$-algebra $A$ and a $G$-invariant faithful positive linear functional~$\varphi$ on $A$ (for $Q$-systems such a correspondence is explicitly stated in~\cite{arXiv:1511.07982}*{Proposition~3.4}). Then Theorem~\ref{thm:standard} and Proposition~\ref{prop:standard-iso} for $\CC=\Rep G$ translate into the following.

\begin{theorem} \label{thm:standard2}
For any finite dimensional (left or right) $G$-C$^*$-algebra $A$ there exists a unique $G$-invariant faithful state $\varphi$ on $A$ such that if we define a scalar product on $A$ using $\varphi$, then for the product map $m\colon A\otimes A\to A$ we have $mm^*=(\dim_q A)\iota$.
\end{theorem}

\bp By passing from $A$ to $A^\opos$ if necessary, we may assume
that $A$ is a right $G$-C$^*$-algebra, see
Remark~\ref{rem:left-right}. Then the existence of $\varphi$ is
equivalent to the statement that any C$^*$-Frobenius algebra in
$\Rep G$ is isomorphic to a standard $Q$-system such that the unit
$v\colon\un\to A$ is an isometry. If $\varphi'$ is another state as
in the formulation, then by Proposition~\ref{prop:standard-iso} the
identity map $A\to A$ must be unitary with respect to the scalar
products defined by $\varphi$ and $\varphi'$, so $\varphi=\varphi'$.
\ep

We call the state $\varphi$ given by the above theorem the \emph{canonical invariant state} on $A$, and will usually denote it by~$\varphi_A$. Unless stated otherwise, we will also always assume that $A$ is equipped with the scalar product defined by~$\varphi_A$.

\smallskip

Assume $A$ is a finite dimensional right $G$-C$^*$-algebra. If $X$
is a finitely generated right Hilbert $A$-module, then we can turn
it into a finite dimensional unitary representation of $G$ with
respect to the scalar product $(x,y)=\varphi_A(\langle
y,x\rangle_A)$. Any covariant representation of the pair $(A,G)$ on
a finite dimensional Hilbert space arises this way. By
Lemma~\ref{lem:FrobeniusVScstar} this allows us to identify the
category $\DD_A=\rmodcatin{A}{G}$ of finitely generated right
Hilbert $A$-modules with the category $\rmodcatin{A}{\Rep G}$ of
right unitary $A$-modules in $\Rep G$, so from now on we will only
use the lighter notation $\rmodcatin{A}{G}$. Of course, for this
identification any faithful $G$-invariant state on $A$ can be used,
but $\varphi_A$ is the most natural choice. Similarly, we identify
$\bimodcatin{A}{G}$ with $\bimodcatin{A}{\Rep G}$ and
$\lmodcatin{A}{G}$ with $\lmodcatin{A}{\Rep G}$.

\smallskip

The existence of canonical invariant states is not difficult to establish directly, without relying on Theorem~\ref{thm:standard}. In order to see this, assume as above that $A$ is a finite dimensional C$^*$-algebra with an action $\alpha\colon A\to A\otimes C(G)$ of $G$. Denote by $\lambda(a)$ the operator of multiplication on the left by $a\in A$. Consider the representation $\pi_\alpha$ of the algebra ${\mathcal U}(G)$ on the space $A$ given by $\pi_\alpha(\omega)a=\omega\rhd a=(\iota\otimes\omega)\alpha(a)$. Recall also that we denote by $\rho\in{\mathcal U}(G)$ the Woronowicz character.

\begin{proposition}\label{prop:invstate}
For any finite dimensional right $G$-C$^*$-algebra $A$, we have
$$
\varphi_A(a)=(\dim_q A)^{-1}\Tr(\lambda(a)\pi_\alpha(\rho))\ \ \text{for all}\ \ a\in A.
$$
\end{proposition}

\bp It is well known that $(\dim_q A)^{-1}\Tr(\lambda(\cdot)\pi_\alpha(\rho))$ is an invariant state. Therefore we only have to show that if we define a scalar product on $A$ using this state, then $mm^*=(\dim_q A)\iota$, or equivalently, if we define a scalar product using the positive linear functional $\Tr(\lambda(\cdot)\pi_\alpha(\rho))$, then $mm^*=\iota$.

Consider the one-parameter group of automorphisms $\beta_t$ of $A$ defined by $\beta_t(a)=\rho^{it}\rhd a$. Since $A$ is finite dimensional, there exists a positive invertible element $b\in A$ such that $\beta_t=\Ad b^{it}$. Then $\pi_\alpha(\rho)=\Ad b$. Since $A$ is a direct sum of full matrix algebras, it is then enough to show the following: if $c\in\Mat_n(\C)$ is a positive invertible matrix and we define a scalar product on $\Mat_n(\C)$ using the positive linear functional $a\mapsto \Tr(\lambda(a)(\Ad c))=\Tr(ac)\Tr(c^{-1})$, then for the product $m$ on $\Mat_n(\C)$ we have $mm^*=\iota$. But this is a straightforward computation.
\ep

\begin{remark} \label{rem:altinv}
The above expression for $\varphi_A$ can be interpreted as follows. Consider the unique $G$-invariant conditional expectation $E\colon A\to A^G$. The elements of $A^G$ act by left multiplication on $A$, and this way we can consider $A^G$ as a subalgebra of the endomorphism ring of the object $A\in\Rep G$. Hence the normalized categorical trace $\tr_A$ on this endomorphism ring defines a tracial state on $A^G$.  (To be more precise, in order to consider $A$ as an object of $\Rep G$, we have to define a scalar product on $A$ using an invariant faithful state. But the trace on $A^G$ that we get this way is independent of any choices.) Then $\varphi_A=\tr_A E$.
\end{remark}

\begin{remark}
\label{rem:KMS-for-can-inv-stt} Another consequence of the
proposition is that $\varphi_A$ satisfies the KMS condition with
respect to $(\beta_t)_t$. Indeed, we have $\pi_\alpha(\rho)^z
\lambda(a) \pi_\alpha(\rho)^{-z} = \lambda(\rho^z \rhd a)$ and
$$
\varphi_A(a b) = \dim_q(A)^{-1} \Tr(\lambda(a b) \pi_\alpha(\rho)) = \dim_q(A)^{-1} \Tr(\lambda(b) \pi_\alpha(\rho) \lambda(a)  \pi_\alpha(\rho)^{-1} \pi_\alpha(\rho)) = \varphi_A(b \beta_{-i}(a)).
$$
\end{remark}

\section{Invertible bimodule categories}\label{sec:invertible}

\subsection{Invertible bimodule categories and Morita--Galois objects}

The notion of an invertible bimodule category was introduced
in~\cite{MR2677836}. Since relative tensor product of module
categories over infinite C$^*$-tensor categories requires some
discussion, we will adopt the following definition, which is
equivalent to the unitary version of the one in \cite{MR2677836} for finite rigid C$^*$-tensor
categories. We will return to relative tensor products in
Section~\ref{sec:relten}.

\begin{definition} \label{def:invert}
A nonzero $\CC_1$-$\CC_2$-module category $\DD$ over rigid
C$^*$-tensor categories $\CC_1$ and $\CC_2$ is called
\emph{invertible} if there exists a rigid C$^*$-$2$-category $\CC$
with the set $\{1,2\}$ of $0$-cells such that $\CC_{11}$ is
unitarily monoidally equivalent to $\CC_1$, $\CC_{22}$ is unitarily
monoidally equivalent to $\CC_2$, and $\CC_{12}$ is unitarily
equivalent to~$\DD$ as a $\CC_1$-$\CC_2$-module category.
\end{definition}

Invertible bimodule categories can be defined more intrinsically
without mentioning $2$-categories. For this we need to recall a few
definitions.

Let $\CC$ be a rigid C$^*$-tensor category and $\DD$ be a right $\CC$-module category. Then $\DD$ is called \emph{indecomposable} if it is not equivalent to a direct sum of two nonzero module categories. If $\DD$ is semisimple as a C$^*$-category, then $\DD$ indecomposable if and only if every nonzero object $X\in\DD$ is generating, meaning that any other object is a subobject of $XU$ for some $U\in\CC$.

The action of $\cC$ on $\cD$ is called \emph{proper}, or
\emph{cofinite}~\cite{arXiv:1511.07982}, if for any $X, Y\in\DD$ we
have $\cD(X, Y U_i) = 0$ for all but finitely many $i$, where
$(U_i)_i$ are representatives of the isomorphism classes of simple
objects in $\cC$. Note that if $\cD$ is indecomposable this can be
relaxed to $\cD(X, Y U_i) = 0$ for all but finitely many $i$, for
\emph{some} nonzero $X, Y$.

Finally, recall that we denote by $\End_\cC(\cD)$ the
C$^*$-tensor category with objects the unitary $\cC$-module
endofunctors of $\cD$ and morphisms the uniformly bounded natural
transformations between such endofunctors. The purely algebraic
counterpart of this category is also denoted by $\cC_\cD^*$.

\begin{theorem}
\label{thm:inv-bimod-cat-is-semisimple-proper-mod-cat}
Let $\CC_1$ and $\CC_2$ be rigid C$^*$-tensor categories and $\DD$ be a nonzero $\CC_1$-$\CC_2$-module category. Then~$\DD$ is invertible if and only if the following conditions are satisfied:
\begin{itemize}
\item[(a)] $\DD$ is semisimple as a C$^*$-category;
\item[(b)] the action of $\CC_2$ on $\DD$ is proper;
\item[(c)] the functor $\CC_1\to\End_{\CC_2}(\DD)$ defined by the action of $\CC_1$ on $\DD$ is an equivalence of C$^*$-tensor categories.
\end{itemize}
Furthermore, if these conditions are satisfied, then $\DD$ is indecomposable as a left $\CC_1$-module category and as a right $\CC_2$-module category.
\end{theorem}

\bp Assume first that $\DD$ is invertible. By passing to equivalent
categories we may assume that we have a rigid C$^*$-$2$-category
$\CC$ with the set $\{1,2\}$ of $0$-cells such that
$\CC_{11}=\CC_1$, $\CC_{22}=\CC_2$ and $\DD=\CC_{12}$. Condition~(a)
is satisfied as $\CC$ is rigid. Condition (b) is also satisfied,
since if $\DD(X,YU_i)\ne0$ then $U_i$ appears in the decomposition
of $\bar Y X\in \CC_2$ into a direct sum of simple objects. It
remains to check (c).

Let us fix a nonzero $X \in \cC_{12}$, and let $F$ be a
$\cC_2$-module endofunctor on $\cC_{12}$. Putting $Y = F(X)$, for
any object $Z \in \cC_{12}$ the isometry $d(X)^{-1/2}
F(\bar{R}_X\otimes\iota_Z)$ induces a realization of $F(Z)$ as a
direct summand of $F(X \bar{X} Z) \cong Y \bar{X} Z$. Thus, $F$ is a
direct summand of $Y \bar{X} \in \cC_1$ in $\End_{\cC_2}(\cC_{12})$.
It follows that we just need to show
$$
\cC_1(U, U') \cong \Mor_{\End_{\cC_2}(\cC_{12})}(U, U')\ \ \text{for
all}\ \ U,\,U'\in\CC_1.
$$
Thus, suppose that $(\eta_Z \colon U Z \to U' Z)_Z$ is a natural transformation of $\cC_2$-module functors. This means we have $\eta_{Z V} = \eta_Z \otimes \iota_V$ for $Z \in \cC_{12}$ and $V \in \cC_2$. We claim that $\eta_0 = (\iota \otimes \tr_X)(\eta_X) \in \cC(U, U')$ satisfies $\eta_Z = \eta_0 \otimes \iota_Z$ for $Z \in \cC_{12}$. Indeed, we have
$$
(\iota \otimes \tr_X)(\eta_X) \otimes \iota_Z = \frac{1}{d(X)}
(\iota_{U'} \otimes \bar{R}_X^* \otimes \iota_Z) (\eta_X \otimes
\iota_{\bar{X}Z}) (\iota_{U} \otimes \bar{R}_X \otimes \iota_Z)
$$
but $\eta_X \otimes \iota_{\bar{X}Z} = \eta_{X \bar{X} Z}$ and the
naturality of $\eta$ implies the right hand side of the above
identity is equal to $d(X)^{-1} \eta_Z (\iota_U\otimes\bar{R}_X^*
\bar{R}_X\otimes\iota_Z) = \eta_Z$.

Finally, $\DD$ is indecomposable as a right $\CC_2$-module category,
since for any nonzero objects $X,Y\in\DD$, the object $Y$ is a
subobject of $X(\bar XY)$. Similarly, $\DD$ is indecomposable as a
left $\CC_1$-module category.

\smallskip

Conversely, suppose that a nonzero $\cC_1$-$\cC_2$-module category $\cD$
satisfies conditions (a--c). By
Theorem~\ref{thm:proper-mod-equiv-to-Q-sys-mod}, there is an
irreducible $Q$-system $A$ in $\cC_2$ such that $\cD$ is unitarily
equivalent to $\lmodcatin{A}{\cC_2}$ as a $\cC_2$-module category.
We can therefore consider $\lmodcatin{A}{\cC_2}$ as a
$\CC_1$-$\CC_2$-module category equivalent to $\DD$. On the other
hand, see Section~\ref{sec:Q}, there is a rigid C$^*$-$2$-category
$(\cC'_{ij})_{i, j}$ such that $\cC'_{12}=\lmodcatin{A}{\cC_2}$ and
$\cC_2 = \cC'_{22}$. By the first part of the proof, the action of
$\cC'_{11}$ on $\cC'_{12}$ defines an equivalence between
$\cC'_{11}$ and $\End_{\cC_2}(\cC'_{12})$. The condition (c) implies
then that~$\cC_1$ is unitarily monoidally equivalent to $\cC'_{11}$
in a compatible way with respect to the $\CC_1$-$\CC_2$-
 and $\cC'_{11}$-$\CC_2$-module category structures on $\cC'_{12}$. Thus,
$\cD$ is an invertible bimodule category. \ep

From the above proof we also see that if $\CC$ is a rigid
C$^*$-tensor category and $\DD$ is a nonzero semisimple
indecomposable right $\CC$-module category such that the action of
$\CC$ on $\DD$ is proper, then $\End_\CC(\DD)$ is a rigid
C$^*$-tensor category and the $\End_\CC(\DD)$-$\CC$-module category
$\DD$ is invertible, so that the $2$-category structure and rigidity
automatically follow from the one-sided module structure.

\medskip

We now turn to representation categories of compact quantum groups.
Our goal is to find an algebraic characterization of invertibility
of bimodule categories. Throughout the rest of this section $G_1$
and $G_2$ denote compact quantum groups, and $A$ denotes a unital
C$^*$-algebra. We also fix representatives of irreducible
classes~$(U_i)_i$ and~$(V_j)_j$ in $\Rep G_1$ and $\Rep G_2$
respectively. Following our conventions in Section~\ref{sec:conv},
their underlying Hilbert spaces are denoted by~$H_i$ and~$H_j$.

\begin{definition}
Let $G_1$ and $G_2$ be reduced compact quantum groups. A
$G_1$-$G_2$-\emph{Morita--Galois object} is a unital C$^*$-algebra
$A$ together with commuting free actions $\alpha_1\colon A\to
C(G_1)\otimes A$ and $\alpha_2\colon A\to A\otimes C(G_2)$ such that
there is a $G_1$-$G_2$-equivariant isomorphism
$$
A^{G_1}\otimes\A\cong A^{G_2}\otimes\A
$$
of $A^{G_1}\otimes A^{G_2}$-$\A$-modules.
\end{definition}

A few comments are in order. The subspace
$$
\A = \{a \in A \mid \alpha_1(a) \in \C[G_1] \otimes A,\ \alpha_2(a)
\in A \otimes \C[G_2] \}
$$
is the regular subalgebra of $A$ with respect to the joint action of $G_1$ and $G_2$, and the tensor product is the algebraic one (we will later see that these assumptions force $A^{G_1}$ and $A^{G_2}$ to be finite dimensional). The left $A^{G_1}\otimes A^{G_2}$-module structure on $A^{G_1}\otimes\A$ is given by
$$
(a\otimes b)(x\otimes y)=a x\otimes b y,
$$
while on $A^{G_2}\otimes \A$ by
$$
(a\otimes b)(x\otimes y)=b x\otimes a y.
$$
Since the actions of $G_1$ and $G_2$ commute, $A^{G_1}$ is a right
$G_2$-C$^*$-algebra, in particular, a right $C(G_2)$-comodule, so
$A^{G_1}\otimes\A$ is a right $C(G_2)$-comodule. Similarly,
$A^{G_2}\otimes \A$ is a left $C(G_1)$-comodule. We therefore
require the isomorphism $A^{G_1}\otimes\A\cong A^{G_2}\otimes\A$ to
respect this comodule structures.

\smallskip

Existence of an isomorphism $A^{G_1}\otimes\A\cong A^{G_2}\otimes\A$ as in the above definition can be reformulated as a compatibility condition on Frobenius algebra structures on $A^{G_1}$ and $A^{G_2}$. Namely, we have the following result.

\begin{proposition} \label{prop:equivd}
Assume we are given commuting actions $\alpha_1\colon A\to C(G_1)\otimes A$ and $\alpha_2\colon A\to A\otimes C(G_2)$. Then a $G_1$-$G_2$-equivariant isomorphism $A^{G_1}\otimes\A\cong A^{G_2}\otimes\A$
of $A^{G_1}\otimes A^{G_2}$-$\A$-modules exists if and only if the following conditions hold:
\begin{itemize}
\item[(a)] the fixed point algebras $A^{G_1}$ and $A^{G_2}$ are finite dimensional;
\item[(b)] there exist a faithful $G_2$-invariant state $\psi_1$ on $A^{G_1}$
and a faithful $G_1$-invariant state~$\psi_2$ on~$A^{G_2}$ such that
if $m^*_1(1)=x^i\otimes x_i$, where $m_1\colon A^{G_1}\otimes
A^{G_1}\to A^{G_1}$ is the product map and the adjoint is computed
with respect to $\psi_1 \otimes \psi_1$ and $\psi_1$, and similarly
$m^*_2(1)=y^j\otimes y_j$ for the product $m_2$ and the
state~$\psi_2$ on~$A^{G_2}$, then we have
$$
x^i y^j x_i \otimes y_j=\lambda 1\otimes1\ \ \text{and}\ \ y^j x^i y_j\otimes x_i=\lambda1\otimes1
$$
for a nonzero scalar $\lambda$.
\end{itemize}
Furthermore, if these conditions are satisfied, then
\begin{itemize}
\item[(i)] the map
$$
S\colon A^{G_1}\otimes\A\to A^{G_2}\otimes\A,\quad a\otimes b\mapsto y^j\otimes a y_j b,
$$
is a $G_1$-$G_2$-equivariant isomorphism of $A^{G_1}\otimes A^{G_2}$-$\A$-modules, with the inverse given by $e\otimes f\mapsto \lambda^{-1} x^i\otimes ex_if$;
\item[(ii)] as the states $\psi_1$ and $\psi_2$ we can take the canonical invariant states, in which case $\lambda=\dim_q A^{G_1}=\dim_q A^{G_2}$, where we consider $A^{G_1}$ as a $G_2$-module and $A^{G_2}$ as a $G_1$-module;
\item[(iii)] the relative commutants $(A^{G_1})'\cap A^{G_2}$ and $(A^{G_2})'\cap A^{G_1}$ are trivial; in particular, $A^{G_1}$ is a simple $G_2$-algebra and $A^{G_2}$ is a simple $G_1$-algebra.
\end{itemize}
\end{proposition}

\begin{remark}\label{rem:mcb}
The identities in~(b) can be equivalently expressed as
\begin{equation} \label{eq:Mkey}
x^i y x_i=\lambda\psi_2(y)1, \quad y^j x y_j=\lambda\psi_1(x)1 \quad
(x\in A^{G_1}, y\in A^{G_2}).
\end{equation}
Indeed, we may, and will now and in the proof below, assume that the
vectors $x_i$ form a basis. Then the vectors $x^i$ are characterized
by $\psi_1(x_ix^k)=\delta_{ik}$, see Section~\ref{sec:Frobenius},
and similarly for the $y_j$'s. Then $1\otimes 1=\psi_2(y^j)1\otimes
y_j$, so the first identity in (b) is equivalent to
$x^iy^jx_i=\lambda\psi_2(y^j)1$ for all $j$.
\end{remark}

\bp[Proof of Proposition~\ref{prop:equivd}]
Assume we are given a $G_1$-$G_2$-equivariant isomorphism $T\colon A^{G_1}\otimes\A\to A^{G_2}\otimes\A$ of $A^{G_1}\otimes A^{G_2}$-$\A$-modules. By the $G_2$-equivariance it maps $1\otimes 1$ into a vector $z^k\otimes z_k\in A^{G_2}\otimes A^{G_2}$, and then $T$, being a morphism of $A^{G_1}$-$\A$-modules, must be given by
$$
T(a\otimes b)= z^k\otimes a z_k b.
$$
We may assume that the vectors $z^k$ are linearly independent.  It
follows then that they form a basis in $A^{G_2}$. In particular,
$A^{G_2}$ is finite dimensional, and for similar reasons $A^{G_1}$
is finite dimensional as well, which proves~(a).

Using that~$T$ is a morphism of left $A^{G_2}$-modules, we also see that $z^k\otimes z_k$ must be a central vector in the $A^{G_2}$-bimodule $A^{G_2}\otimes A^{G_2}$. Generally speaking, if $e^r_{st}$ are matrix units in $A^{G_2}$, any central vector in $A^{G_2}\otimes A^{G_2}$ has the form $\xi=\sum_{r,s,t}e^r_{st}v\otimes e^r_{ts}$ for a uniquely defined $v\in A^{G_2}$. Moreover, the slices $(\iota\otimes\omega)(\xi)$ for $\omega \in (A^{G_2})^*$ span~$A^{G_2}$ if and only if $v$ is invertible.

Now, take any faithful $G_1$-invariant state $\psi_2$ on $A^{G_2}$
and write $m_2^*(1)=y^j\otimes y_j$ with respect to $\psi_2$. By the
above discussion, the map
$$
S\colon A^{G_1}\otimes\A\to A^{G_2}\otimes\A, \ \ a\otimes b\mapsto y^j\otimes a y_j b,
$$
has the form $S(\zeta)=T(\zeta)(v\otimes1)$ for some invertible element $v\in A^{G_2}$. Hence $S$ is an isomorphism of $A^{G_1}\otimes A^{G_2}$-$\A$-modules. It is easy to see that this isomorphism is $G_1$-$G_2$-equivariant (note that the vector $m_2^*(1)$ is $G_1$-invariant, as $m_2$ is a $G_1$-equivariant map and the state $\psi_2$ is invariant).

Consider now the inverse map $S^{-1}$. By the same considerations as above, if we fix a faithful $G_2$-invariant state~$\tilde\psi_1$ on~$A^{G_1}$ and write $m_1^*(1)=\tilde x^i\otimes \tilde x_i$ with respect to $\tilde\psi_1$, then $S^{-1}$ has the form
$$
S^{-1}(e\otimes f)=\tilde x^i u \otimes e\tilde x_i f \quad (e \in A^{G_2}, f \in \A)
$$
for an invertible element $u\in A^{G_1}$. Let $\psi$ denote the linear functional $\tilde\psi_1(\cdot\, u^{-1})$ on $A^{G_1}$, and note that $\psi(\tilde x_i\tilde x^k u)=\delta_{ik}$. Then we have
$$
\tilde{x}^i u \otimes y^j \tilde{x}_i y_j  = S^{-1}(S(1 \otimes 1)) = 1 \otimes 1  \quad \text{and} \quad y^j \otimes \tilde{x}^i u y_j \tilde{x}_i = S(S^{-1}(1 \otimes 1)) = 1 \otimes 1.
$$
As in Remark~\ref{rem:mcb}, this is equivalent to
\begin{equation}\label{eq:Mkey2}
\tilde x^i u y\tilde x_i=\psi_2(y)1\ \ \text{and}\ \ y^j x
y_j=\psi(x)1 \quad (x \in A^{G_1}, y\in A^{G_2}).
\end{equation}

We may assume that the vectors $y^j$ form an orthonormal basis in
$A^{G_2}$. Then $y_j=y^{j*}$, and as $\psi(x)1=\sum_j y^j x y^{j*}$,
we conclude that $\psi$ is a $G_2$-invariant positive linear
functional. As the pairing defined by $\psi$ is nondegenerate, this
functional is faithful. Put $\lambda=\psi(1)$ and
$\psi_1=\lambda^{-1}\psi$. If we use $\psi_1$ to define $m_1^*$, we
get $m_1^*(1) = \lambda \tilde{x}^i u \otimes \tilde{x}_i$. Together
with \eqref{eq:Mkey2} this shows that condition~(b) is satisfied.
Note also that we have proved (i) along the way.

\smallskip

Conversely, if (a) and (b) are satisfied, then we get the required
structure isomorphism for a $G_1$-$G_2$-Morita--Galois object, as
described in (i).

\smallskip

Next, for (ii), from the above considerations we see that as the
state $\psi_2$ we can take any faithful $G_1$-invariant state. If we
take the canonical invariant state, then from $y^j y_j=\lambda1$ we
get $\lambda=\dim_q A^{G_2}$. But then the identity
$x^ix_i=\lambda1$ implies that $\dim_q A^{G_1}\le\lambda=\dim_q
A^{G_2}$. Similarly, if we start with the canonical invariant state
on~$A^{G_1}$, we get $\dim_q A^{G_1}\ge\dim_q A^{G_2}$. Therefore
$\dim_q A^{G_1}=\dim_q A^{G_2}$ and if we take the canonical
invariant state on one algebra, then we have to take the canonical
invariant state on the other algebra as well.

\smallskip

Finally, for (iii), if $y\in (A^{G_1})'\cap A^{G_2}$, then
$$
\lambda y=x^i x_i y=x^i y x_i=\lambda\psi_2(y)1,
$$
so $y$ is scalar. In particular, there are no non-scalar $G_1$-invariant elements in the center of $A^{G_2}$, so $A^{G_2}$ is a simple $G_1$-algebra. Similarly, $(A^{G_2})'\cap A^{G_1}=\C 1$ and $A^{G_1}$ is a simple $G_2$-algebra.
\ep

\begin{remark}
\label{rem:assoc-isometry-rem1} Another consequence
of~\eqref{eq:Mkey} is that, if we define $A$-valued inner products
on $A^{G_i} \otimes \A$ by $\langle b \otimes a, b' \otimes
a'\rangle_A= \psi_i(b^* b') a^* a'$, then the map $S$ of (i) becomes a
scalar multiple of a unitary. Indeed, for $a, a' \in \A$ and $b, b'
\in A^{G_1}$, we have
$$
\langle y^j \otimes b y_j a, y^k \otimes b' y_k a'\rangle_A =
\psi_2(y^{j*} y^k) a^* y_j^* b^* b' y_k a'.
$$
Since we can arrange $y^k = y_k^*$ as in the above proof, the right
hand side equals
$$
\psi_2(y_j y^k) a^* y^j b^* b' y_k a' = a^* y^j b^* b' y_j a' =
\lambda \psi_1(b^* b') a^* a' = \lambda \langle b \otimes a, b'
\otimes a'\rangle_A.
$$
In particular, $S$ and its inverse extend to isomorphisms of
equivariant right Hilbert $A$-modules. Conversely, starting from a
$G_1$-$G_2$-C$^*$-algebra $A$, if we assume that the actions are
free, $A^{G_1}$ and $A^{G_2}$ are finite dimensional, and that there
is an isomorphism of equivariant $(A^{G_1}\otimes
A^{G_2})$-$A$-correspondences $A^{G_1} \otimes A \to A^{G_2} \otimes
A$, then taking the regular parts, we can verify the Morita--Galois
conditions for $A$.
\end{remark}

The following is our main result.

\begin{theorem}\label{thm:main}
Assume that  we are given commuting actions $\alpha_1\colon A\to
C(G_1)\otimes A$ and $\alpha_2\colon A\to A\otimes C(G_2)$ of
reduced compact quantum groups $G_1$ and $G_2$ on a unital
C$^*$-algebra $A$. Consider the corresponding category~$\DD_A$ of
finitely generated $G_1$-$G_2$-equivariant right Hilbert
$A$-modules. Then the $(\Rep G_2)$-$(\Rep G_1)$-module
category~$\DD_A$ is invertible if and only if $A$ is a
$G_1$-$G_2$-Morita--Galois object.
\end{theorem}

By the Tannaka--Krein type correspondence for quantum group actions
discussed in Section~\ref{sec:TK}, we then get the following
corollary.

\begin{corollary}
For any reduced compact quantum groups $G_1$ and $G_2$, there is a bijective correspondence between the equivalence classes of invertible $(\Rep G_2)$-$(\Rep G_1)$-module categories and the $G_1$-$G_2$-equivariant Morita equivalence classes of $G_1$-$G_2$-Morita--Galois objects. We also have a bijective correspondence between the equivalence classes of pairs $(\DD,X)$, consisting of an invertible $(\Rep G_2)$-$(\Rep G_1)$-category $\DD$ and a nonzero object $X\in\DD$, and the isomorphism classes of $G_1$-$G_2$-Morita--Galois objects.
\end{corollary}

We divide the proof of the theorem into several parts.

\subsection{From invertible bimodule categories to bi-Morita--Galois objects}

Assume first that $\DD_A$ is invertible. Consider the corresponding
rigid C$^*$-$2$-category $\CC$ with the set $\{1,2\}$ of $0$-cells
such that $\cC_1=\CC_{11}$ is equivalent to $\Rep G_1$,
$\cC_2=\CC_{22}$ is equivalent to $\Rep G_2$, and $\DD_A$ is
equivalent to $\CC_{21}$ as a $\CC_2$-$\CC_1$-module category. In
order to simplify the exposition we are not going to distinguish
between $\cC_i$ and $\Rep G_i$, although to be pedantic we should
either explicitly use our fixed unitary monoidal equivalences $\Rep
G_i\to\cC_i$ in all the formulas below or work with bicategories
instead of $2$-categories, that is, assume that $\CC$ has nontrivial
associativity morphisms.

Let $X\in\CC_{21}$ be the object corresponding to $A$. From now on
we will think of $A$ as the result of the construction of a
$G_1$-$G_2$-C$^*$-algebra from the pair $(\cC_{21},X)$. We will see
that the required isomorphism $A^{G_1}\otimes\A\cong A^{G_2}\otimes
\A$ follows from the equality $(X \bar X) X=X(\bar XX)$ in
$\cC_{21}$, while freeness of the actions follows from the
indecomposability of $\cC_{21}$ as a one-sided module category.

\smallskip

We start by establishing the freeness. The regular subalgebra
$\A\subset A$ is
$$
\bigoplus_{\substack{i\in \Irr (G_1) \\ j \in \Irr (G_2)}} \bar{H}_i
\otimes \bar{H}_j \otimes \cC_{21}(X, V_j X U_i).
$$
Recall that $C(G_1)$ coacts on the left by $(U_i^c)^*_{21}$, while
$C(G_2)$ coacts on the right by $V_j^c$.

By construction, we have
\begin{align}
\label{eq:A-G2-and-A-G1-cat-pres} A^{G_2} &= \bigoplus_{i\in \Irr
(G_1)} \bar{H}_i \otimes \cC_{21}(X, X U_i),& A^{G_1} &=
\bigoplus_{j \in \Irr (G_2)} \bar{H}_j \otimes \cC_{21}(X, V_j X).
\end{align}
In other words, the fixed point algebras are the algebras
corresponding to the object $X$ in the category $\cC_{21}$ regarded
as a one-sided module category either over $\Rep G_1$ or over $\Rep
G_2$. The joint fixed point subalgebra $(A^{G_1})^{G_2}$ is
isomorphic to $\cC_{21}(X) = \cC_{21}(X, X)$, so $G_1$ and $G_2$ act
jointly ergodically if and only if $X$ is simple.

Consider a unitary equivalence $\DD_A\to\cC_{21}$ of bimodule
categories provided by the Tannaka--Krein correspondence for
actions. Up to a natural isomorphism, it is described by the
following properties, see~\cite{MR3426224}*{Section~3}. For
$U\in\Rep G_1$ and $V\in\Rep G_2$, we put
$$
F(H_U\otimes H_V\otimes A)=VXU,
$$
and take the morphism $VF(A)U\to F(H_U\otimes H_V\otimes A)$
required by the definition of a bimodule functor to be the identity.
To describe the action of $F$ on morphisms, consider a morphism
$T\colon X\to V_jXU_i$ in $\cC_{21}$. There is a unique morphism
$\tilde T\colon A\to H_i\otimes H_j\otimes A$ in $\DD_A$ mapping
$1\in A$ into
$$
\sum_{\alpha,\beta}\xi_\alpha\otimes\zeta_\beta\otimes\bar\xi_\alpha\otimes\bar\zeta_\beta
\otimes T\in H_i\otimes H_j\otimes \bar H_i\otimes\bar
H_j\otimes\cC_{21}(X,V_jXU_i)\subset H_i\otimes H_j\otimes\A,
$$
where $(\xi_\alpha)_\alpha$ and $(\zeta_\beta)_\beta$ are
orthonormal bases in $H_i$ and $H_j$ respectively. Then we require $F(\tilde T) = T$.

\begin{lemma}\label{lem:fixed-iso}
The strict $(\Rep G_2)$-module functor $\DD_A\to\DD_{A^{G_1}}$,
$Y\mapsto Y^{G_1}$, and strict $(\Rep G_1)$-module functor
$\DD_A\to\DD_{A^{G_2}}$, $Y\mapsto Y^{G_2}$, are equivalences of
categories.
\end{lemma}

\bp We will only prove the first statement. Denote the functor
$\DD_A\to\DD_{A^{G_1}}$, $Y\mapsto Y^{G_1}$, by $E$. As we already
observed above, the $G_2$-C$^*$-algebra $A^{G_1}$ corresponds to the
$(\Rep G_2)$-module category $\CC_{21}$ and object~$X$. It follows
that similarly to the functor $F\colon\DD_A\to\cC_{21}$ we have a
$(\Rep G_2)$-module functor $\tilde
F\colon\DD_{A^{G_1}}\to\CC_{21}$, $\tilde F(H_V\otimes A^{G_1})=VX$,
defining an equivalence of categories. We obviously have $E\tilde
F=F$ on the full subcategory of $\DD_A$ consisting of the modules
$H_V\otimes A$. Since $X$ generates $\CC_{21}$ as a $(\Rep
G_2)$-module category and both $F$ and $\tilde F$ are equivalences
of categories, it follows that $E$ is an equivalence of categories
as well. \ep

\begin{lemma}
The actions of $G_1$ and $G_2$ on $A$ are separately free.
\end{lemma}

\bp Let us only prove freeness of the action of $G_2$. By
Proposition~\ref{prop:freeness} it suffices to show that
$Y^{G_2}\ne0$ for any nonzero $G_2$-equivariant finitely generated
right Hilbert $A$-module $Y$. Furthermore, the proof of that
proposition respects the additional action of $G_1$ on $A$. In other
words, if the action of $G_2$ is not free, then the proof shows that
there exists a nonzero $Y\in\DD_A$ such that $Y^{G_2}=0$. But
this contradicts the previous lemma. \ep

Let us now study the fixed point algebra $A^{G_1}$ in more detail.
Consider the object $X\bar X\in\Rep G_2$. It has the structure of a
standard $Q$-system given by
$$
m = d(X)^{1/2} (\iota_X \otimes R_X^* \otimes \iota_{\bar X}),\ \  v =
d(X)^{-1/2} \bar{R}_X.
$$
In other words, if we use the picture of right unitary
$C(G_2)$-comodules for $\Rep G_2$, we can view $X\bar X$ as a right
$G_2$-C$^*$-algebra with the scalar product defined by the canonical
invariant state. It can be reconstructed from the left $(\Rep
G_2)$-module category $\rmodcatin{X\bar X}{G_2}$, so we have a
canonical isomorphism
\begin{equation} \label{eq:X-barX}
X\bar X \cong \bigoplus_{j \in \Irr (G_2)} \bar{H}_j \otimes
\Mor_{\rmodcatin{X\bar X}{G_2}}(X\bar X, V_j X\bar X).
\end{equation}
On the other hand, the functor $Y\mapsto Y\bar X$
defines a unitary strict $(\Rep G_2)$-module equivalence between
$\CC_{21}$ and $\rmodcatin{X\bar X}{G_2}$, which is an observation
going back to \cite{MR1966524}*{Proposition~4.5}. Therefore
comparing~\eqref{eq:X-barX} with~\eqref{eq:A-G2-and-A-G1-cat-pres} we get an isomorphism
$\theta\colon X\bar X\to A^{G_1}$ of $G_2$-C$^*$-algebras. If we as
usual equip $A^{G_1}$ with the scalar product defined by the
canonical invariant state, $\theta$ becomes a unitary isomorphism of
the standard $Q$-systems $X\bar X$ and $A^{G_1}$ in $\Rep G_2$.

The particular isomorphism $\theta$ that we have defined has the
following important property. As in the proof of
Lemma~\ref{lem:fixed-iso}, consider a $(\Rep G_2)$-module functor
$\tilde F\colon \DD_{A^{G_1}}\to\CC_{21}$ allowing us to reconstruct
$A^{G_1}$ as in~\eqref{eq:A-G2-and-A-G1-cat-pres}. Composing it with
the functor $\CC_{21}\to\Rep G_2$, $Y\mapsto Y\bar X$, we get a $(\Rep G_2)$-module functor
$\tilde{\tilde F}\colon\DD_{A^{G_1}}\to\Rep G_2$ such that
$\tilde{\tilde F}(H_V\otimes A^{G_1})=VX\bar X$.

\begin{lemma}
For any $V\in\Rep G_2$ and any morphism $S\colon H_V\otimes
A^{G_1}\to A^{G_1}$ in $\DD_{A^{G_1}}$, the following diagram
commutes:
$$
\begin{tikzcd}
VX\bar X \ar[r,"\tilde{\tilde F}(S)"] \ar[d,"\iota_V\otimes\theta"'] & X\bar X\ar[d,"\theta"]\\
H_V\otimes A^{G_1} \ar[r,"S"'] & A^{G_1}
\end{tikzcd}
$$
\end{lemma}

\bp This is an immediate consequence of the definitions, making the
following argument essentially tautological.

It is enough to consider $V=V_j$. Take a morphism $T\colon X\to V_j
X$. Let $\tilde T\colon A^{G_1}\to H_j\otimes A^{G_1}$ be the
morphism in $\DD_{A^{G_1}}$ mapping $1\in A^{G_1}$ into $\sum_\beta
\zeta_\beta\otimes\bar\zeta_\beta\otimes T$, where
$(\zeta_\beta)_\beta$ is an orthonormal basis in~$H_j$. Then by
definition we have $\tilde{\tilde F}(\tilde T)=T\otimes\iota_{\bar
X}$. In terms of decomposition \eqref{eq:X-barX} this means that
$\tilde{\tilde F}(\tilde T)$ maps the unit of~$X\bar X$ into
$\sum_\beta \zeta_\beta\otimes\bar\zeta_\beta\otimes
(T\otimes\iota_{\bar X})\in H_j\otimes X\bar X$. Applying
$\iota_j\otimes\theta$ to the last element we get $\sum_\beta
\zeta_\beta\otimes\bar\zeta_\beta\otimes T=\tilde T(1)$. Therefore
$$
\tilde T\theta=(\iota_j\otimes\theta)\tilde{\tilde F}(\tilde
T)\colon X\bar X\to H_j\otimes A^{G_1}.
$$
Since any morphism $S\colon H_j\otimes A^{G_1}\to A^{G_1}$ in
$\DD_{A^{G_1}}$ has the form $\tilde T^*$ for some $T\colon X\to
V_jX$, this proves the lemma. \ep

Note that this lemma implies that $\theta$ extends to a natural isomorphism
of the $(\Rep G_2)$-module functors~$\tilde{\tilde F}$ and the forgetful functor $\DD_{A^{G_1}}\to\Rep G_2$.

\begin{lemma}\label{lem:m1}
Consider the multiplication map $m_1\colon A^{G_1}\otimes A\to A$.
Then
$$
F(m_1)(\theta\otimes\iota_X)=d(X)^{1/2}(\iota_X\otimes R^*_X)\colon X\bar X X\to X.
$$
\end{lemma}

\bp As in the proof of Lemma~\ref{lem:fixed-iso}, consider
the functor $E\colon\DD_A\to\DD_{A^{G_1}}$. Since $\tilde FE=F$ on the modules
$H_V\otimes A$, it suffices to show that
$$
\tilde F(m_1)(\theta\otimes\iota_X)=d(X)^{1/2}(\iota_X\otimes R^*_X),
$$
where now $m_1$ denotes the multiplication map on $A^{G_1}$. Applying the functor $Y\mapsto Y\bar X$
we have to check that
$$
\tilde{\tilde F}(m_1)(\theta\otimes\iota_{X\bar X})=d(X)^{1/2}(\iota_X\otimes R^*_X\otimes\iota_{\bar X}).
$$
By the previous lemma the left hand side equals
$\theta^* m_1(\theta\otimes\theta)$, which is exactly the right hand side, since $\theta$ is an isomorphism of $Q$-systems.
\ep

This lemma characterizes the unitary isomorphism $\theta$. Indeed, any other isomorphism has the form $\theta u$, where $u$ is a unitary
automorphism of $X\bar X$. Then
$$
(\iota_X\otimes R^*_X)(u\otimes\iota_{\bar X})=\iota_X\otimes R^*_X,
$$
which implies $u=\iota$.

Similar arguments apply to $A^{G_2}$ and $\bar XX$. The main difference is that we have to use the picture of left unitary $C(G_1)$-comodules for $\Rep G_1$, and
since the tensor product of left $C(G_1)$-comodules corresponds to the opposite
tensor product of representations of $\Rep G_1$, we have to replace the product on $A^{G_2}$ with the opposite one in order to get a C$^*$-Frobenius algebra in $\Rep G_1$.
As usual we equip $A^{G_2}$ with the scalar product defined by the canonical invariant state, so $(a,b)=\varphi_{A^{G_2}}(b^*a)$ (where $b^*a$ denotes the original product). Then
$(A^{G_2})^\opos$ becomes a standard $Q$-system in $\Rep G_1$ and we get the following result.

\begin{lemma} \label{lem:m2}
Consider the standard $Q$-system $(\bar XX, m = d(X)^{1/2} (\iota_{\bar X} \otimes \bar R_X^* \otimes \iota_X), v =
d(X)^{-1/2} {R}_X)$ in $\Rep G_1$. Then there exists a unique unitary isomorphism $\theta'\colon\bar XX\to (A^{G_2})^\opos$ of standard $Q$-systems
such that for the product map $m_2\colon A^{G_2}\otimes A\to A$ we have
$$
F(m_2)(\iota_X\otimes\theta')=d(X)^{1/2}(\bar
R^*_X\otimes\iota_X)\colon X\bar XX\to X.
$$
\end{lemma}

We are now ready to establish the key property of the algebra $A$.

\begin{lemma}
There is a $G_1$-$G_2$-equivariant isomorphism $A^{G_1}\otimes A\cong A^{G_2}\otimes A$ of $(A^{G_1}\otimes A^{G_2})$-$A$-modules.
\end{lemma}

\bp Consider the modules $X_1=A^{G_1}\otimes A$ and
$X_2=A^{G_2}\otimes A$. They are $A^{G_1}$-$(A^{G_2})^\opos$-modules
in the category $\DD_A$. Using the isomorphisms $\theta$ and
$\theta'$ we can equivalently view $X_1$ and $X_2$ as $X\bar
X$-$\bar XX$-modules in~$\DD_A$. Then the bimodule functor
$F\colon\DD_A\to\cC_{21}$ allows us to introduce $X\bar X$-$\bar
XX$-module structures on~$F(X_1)$ and $F(X_2)$, hence also on
$(\theta^*\otimes\iota_X)F(X_1)$ and
$(\iota_X\otimes\theta'^*)F(X_2)$. Let us consider them in more
detail.

We have $F(X_1)=A^{G_1}X$. The left $A^{G_1}$-module structure on
$A^{G_1}X$ comes from the multiplication on~$A^{G_1}$. Hence the
left $X\bar X$-module structure on
$(\theta^*\otimes\iota_X)F(X_1)=X\bar X X$ also comes from the
multiplication on $X\bar X$. On the other hand, the right
$(A^{G_2})^\opos$-structure on $A^{G_1}X$ is given by
$$
\iota_{A^{G_1}}\otimes F(m_2)\colon A^{G_1}XA^{G_2}\to A^{G_1}X.
$$
Using Lemma~\ref{lem:m2} we conclude that the right $\bar XX$-module
structure on $(\theta^*\otimes\iota_X)F(X_1)=X\bar X X$ comes from
the multiplication on $\bar X X$. Similar arguments apply to
$(\iota_X\otimes\theta'^*)F(X_2)$.

We thus have the equalities
$$
 (\theta^*\otimes\iota_X)F(X_1)=X\bar X X=(\iota_X\otimes\theta'^*)F(X_2)
$$
of $X\bar X$-$\bar XX$-modules. Hence the unique isomorphism $\pi\colon X_1\to X_2$ in~$\DD_A$ such that
$$
F(\pi)=(\iota_X\otimes{\theta'}) (\theta^*\otimes\iota_X)
$$
must be an isomorphism of $X\bar X$-$\bar XX$-modules, or equivalently, of $A^{G_1}$-$(A^{G_2})^\opos$-modules.
\ep

This finishes the proof of Theorem~\ref{thm:main} in one direction.

\begin{remark}
In the above proof we used a functor $F\colon\DD_A\to\CC_{21}$. We could have equally well used the functor going in the opposite direction defined in~\cite{MR3141721}. Namely, we have a functor mapping $Y\in\CC_{21}$ into a completion of
$$
\cE_Y=\bigoplus_{i,j}\bar H_i\otimes\bar
H_j\otimes\cC_{21}(X,V_jYU_i).
$$
This functor has the obvious action on morphisms. However, its
bimodule functor structure is a bit more difficult to describe. One
minor advantage of using this functor is that we would be able to
compute an isomorphism $A^{G_1}\otimes A\cong A^{G_2}\otimes A$
rather than merely prove its existence.
\end{remark}

Before we turn to the proof of the theorem in the opposite direction, let us finish this section with the following observation.

\begin{proposition}\label{prop:invstates}
The canonical invariant state $\varphi_{A^{G_1}}$ on $A^{G_1}$ is
given by the composition of the $G_2$-invariant conditional
expectation $A^{G_1}\to (A^{G_1})^{G_2}=\cC_{21}(X)$ with the
normalized categorical trace~$\tr_X$ on~$\cC_{21}(X)$. Similarly for
the canonical invariant state on $A^{G_2}$. In particular, there
exists a unique $G_1$-$G_2$-invariant state $\varphi$ on~$A$ such
that its restrictions to $A^{G_1}$ and $A^{G_2}$ coincide with the
canonical invariant states~$\varphi_{A^{G_1}}$
and~$\varphi_{A^{G_2}}$, respectively.
\end{proposition}

We call this $\varphi$ the \emph{canonical invariant state} on $A$.

\bp Take $S\in(A^{G_1})^{G_2}=\cC_{21}(X)$. We have to show that
$\varphi_{A^{G_1}}(S)=\tr_X(S)$. By Remark~\ref{rem:altinv} and
definition of the product in $A^{G_1}$, we have
$$
\varphi_{A^{G_1}}(S)=(\dim_q A^{G_2})^{-1} \sum_j (\dim_q V_j)(\Tr
S_j),
$$
where $S_j$ is the operator on the vector space $\cC_{21}(X,V_jX)$
given by $T\mapsto (\iota\otimes S)T$. By the Frobenius reciprocity we
can identify $\cC_{21}(X,V_jX)$ with $\cC_2(\bar V_j,X\bar X)$. In
this picture the operator $S_j$ becomes $T\mapsto (S\otimes\iota)T$.
But now $X\bar X$ is an object in $\Rep G_2$, and since it
decomposes into a direct sum of copies of $\bar V_j$, we get
$$
\sum_j (\dim_q V_j)(\Tr S_j)=\Tr_{X\bar X}(S\otimes\iota_{\bar
X})=\Tr_X(S)d(X)=\tr_X(S)d(X)^2.
$$
This implies that
$$
\varphi_{A^{G_1}}(S)=\tr_X(S)\ \ \text{and}\ \ \dim_q
A^{G_1}=d(X)^2.
$$

The statement for $\varphi_{A^{G_2}}$ is proved similarly. The last
statement of the proposition is now obvious: the unique
$G_1$-$G_2$-invariant state extending $\varphi_{A^{G_1}}$ and
$\varphi_{A^{G_2}}$ is given by the composition of the unique
$G_1$-$G_2$-invariant conditional expectation $A\to\cC_{21}(X)$ with
$\tr_X$. \ep

\subsection{From bi-Morita--Galois objects to invertible bimodule categories}

Conversely, assume that $A$ is a $G_1$-$G_2$-Morita--Galois object.
As above, we can consider $A^{G_1}$ as a standard $Q$-system in
$\Rep G_2$. Then we have an invertible $(\Rep
G_2)$-$(\bimodcatin{A^{G_1}}{G_2})$-module category
$\rmodcatin{A^{G_1}}{G_2}$. We will show that the C$^*$-tensor
category $\bimodcatin{A^{G_1}}{G_2}$ is equivalent to $\Rep G_1$ and
the bimodule category $\rmodcatin{A^{G_1}}{G_2}$ is equivalent to
$\DD_A$ in a coherent way.

\begin{lemma} \label{lem:iso}
Let $Y$ (resp.~$Y'$) be a $G_1$-$G_2$-equivariant
$A^{G_1}$-$A$-correspondence (resp.~an equivariant
$A^{G_2}$-$A$-correspondence). We then have a
$G_1$-$G_2$-equivariant unitary isomorphism
$$
Y^{G_1} \otimes Y' \cong Y'^{G_2} \otimes Y
$$
of $A^{G_1}\otimes A^{G_2}$-$A$-correspondences.
\end{lemma}

Note that in this formulation the scalar product on $Y^{G_1}$ is
defined using the $A^{G_1}$-valued inner product and the canonical
invariant state on $A^{G_1}$, and similarly for $Y'^{G_2}$.

\bp Put $\lambda=\dim_q A^{G_1}=\dim_q A^{G_2}$, and let $x_i, x^i,
y_j, y^j$ be as in Proposition~\ref{prop:equivd} (b), where we take
the canonical invariant states. Consider the map
$$
S_0\colon Y^{G_1} \otimes Y'^{G_2} \otimes \A  \to Y'^{G_2} \otimes
Y, \quad \xi \otimes \eta \otimes a \mapsto \lambda^{-1/2} \eta y^j
\otimes \xi y_j a.
$$
Then by the $A^{G_2}$-centrality of $y^j \otimes y_j$, this descends
to a map from  $Y^{G_1} \otimes Y'^{G_2} \otimes_{A^{G_2}} \A$.
Moreover, as the action of $G_2$ on $A$ is free,
Proposition~\ref{prop:imprim} implies $Y'^{G_2} \otimes_{A^{G_2}} \A
\cong \cY'$. Thus, $S_0$ induces a map
$$
S\colon Y^{G_1} \otimes \cY' \to Y'^{G_2} \otimes \cY, \quad \xi
\otimes \eta a \mapsto  \lambda^{-1/2} \eta y^j \otimes \xi y_j a.
$$
Similarly, $T(\eta \otimes \xi a) = \lambda^{-1/2} \xi x^i \otimes
\eta x_i a$ is a well-defined map from $Y'^{G_2} \otimes \cY$ to
$Y^{G_1} \otimes \cY'$.

By the above formulas, $S$ and $T$ are $A^{G_1} \otimes
A^{G_2}$-$\cA$-module morphisms. They are also equivariant with
respect to the actions of $G_1$ and $G_2$, cf.~the proof of
Proposition~\ref{prop:equivd}. It remains to show that they are
inverse to each other. When $\xi \in Y^{G_1}$, $\eta \in Y'^{G_2}$,
and $a \in \A$, we have
$$
S T (\eta \otimes \xi a) = \lambda^{-1} \eta y^j \otimes\xi x^i y_j
x_i a.
$$
Using \eqref{eq:Mkey}, the right hand side is equal to
$\varphi_{A^{G_2}}(y_j) \eta y^j \otimes \xi a = \eta \otimes \xi
a$, which shows $S T = \iota$. A similar computation shows $T S =
\iota$.

Finally, let us show that $S$ is unitary with respect to the $A$-valued inner products.
We have
$$
\langle S(\xi' \otimes \eta' a'), S(\xi \otimes \eta a) \rangle_A =
\lambda^{-1} \varphi_{A^{G_2}}(\langle \eta' y^k, \eta
y^j\rangle_{A^{G_2}}) \langle \xi' y_k a', \xi y_j a \rangle_A.
$$
Using
$$
\varphi_{A^{G_2}}(\langle \eta' y^k, \eta y^j\rangle_{A^{G_2}}) y_j
= \varphi_{A^{G_2}}(y^{k*} \langle \eta', \eta \rangle_{A^{G_2}}
y^j) y_j = y^{k*} \langle \eta', \eta \rangle_{A^{G_2}},
$$
we see that the above expression equals $\lambda^{-1} \langle \xi'
y_k a', \xi y^{k
*} \langle \eta', \eta \rangle_{A^{G_2}} a \rangle_A$. Using $y_k^*
x y^{k *} = \lambda \varphi_{A^{G_1}}(x)$ for $x \in A^{G_1}$ and
that the $A^{G_1}$-valued inner product on $Y^{G_1}$ is the
restriction of the $A$-valued one on $Y$, we arrive at
$\varphi_{A^{G_1}}(\langle \xi', \xi \rangle_{A^{G_1}}) a'^{*} a$,
which is the inner product of $\xi' \otimes \eta' a'$ and $\xi
\otimes \eta a$. This shows the unitarity of $S$. \ep

\begin{corollary} \label{cor:generation}
Any module $X\in\DD_A$ embeds into $H_U\otimes A$ for some $U\in\Rep G_1$, as well as into $H_V\otimes A$ for some $V\in\Rep G_2$.
\end{corollary}

\bp Any $X\in\DD_A$ embeds into $H_W\otimes H_V\otimes A$ for some
$W\in\Rep G_1$ and $V\in\Rep G_2$. Moreover, $H_V$ embeds into
$A^{G_1} \otimes H_V $, and the above lemma implies $A^{G_1} \otimes
H_V \otimes A \cong (H_V \otimes A)^{G_2} \otimes A$. Thus, $X$
embeds into $H_W\otimes (H_V\otimes A)^{G_2}\otimes A$, which proves
the first statement. The second is proved similarly. \ep

Consider now the spectral functor
$$
F\colon\Rep G_1\to\bimodcat{A^{G_1}},\ \ U\mapsto (H_U\otimes A)^{G_1},
$$
defined by the action of $G_1$ on $A$. Since the action is free, it is a unitary tensor functor, with the tensor structure given by
$$
F_2\colon (H_U\otimes A)^{G_1}\otimes_{A^{G_1}}(H_V\otimes A)^{G_1}\to (H_{UV}\otimes A)^{G_1},\ \ (\xi\otimes a)\otimes(\zeta\otimes b)\mapsto (\xi\otimes\zeta)\otimes ab.
$$
Clearly we can view $F$ as a unitary tensor functor $\Rep G_1\to\bimodcatin{A^{G_1}}{G_2}$.

\begin{proposition}
The functor $F\colon\Rep G_1\to\bimodcatin{A^{G_1}}{G_2}$ is an equivalence of categories.
\end{proposition}

\bp
By Lemma~\ref{lem:iso}, for any $V\in\Rep G_2$, we have a $G_2$-equivariant isomorphism
$$
A^{G_1}\otimes H_V\otimes A^{G_1}\cong ((H_V\otimes A)^{G_2}\otimes A)^{G_1}=F((H_V\otimes A)^{G_2})
$$
of $A^{G_1}$-bimodules. This shows that the functor $F$ is dominant, that is, any object of $\bimodcatin{A^{G_1}}{G_2}$ is a subobject of the image of an object of $\Rep G_1$. Since $F$ is also faithful, it remains to show that $F$ is full. It suffices to check that the map
\begin{equation}\label{eq:surjF}
F\colon\Mor(\un,U)\to\Mor_{\bimodcatin{A^{G_1}}{G_2}}(A^{G_1},(H_U\otimes
A)^{G_1})
\end{equation}
is surjective for any $U\in\Rep G_1$. The morphism space on the
right can be identified with the space of $G_2$-invariant
$A^{G_1}$-central vectors in $(H_U\otimes A)^{G_1}$. Since
$$
((H_U\otimes A)^{G_1})^{G_2}=(H_U\otimes A^{G_2})^{G_1},
$$
this space coincides with
$$
(H_U\otimes ((A^{G_1})'\cap A^{G_2}))^{G_1}=(H_U\otimes\C1)^{G_1}=H_U^{G_1}\otimes\C1,
$$
where we used that $(A^{G_1})'\cap A^{G_2}=\C1$ by Proposition~\ref{prop:equivd}(iii). This shows that the map \eqref{eq:surjF} is indeed surjective.
\ep

Consequently, we can view $\rmodcatin{A^{G_1}}{G_2}$ as an invertible $(\Rep G_2)$-$(\Rep G_1)$-module category. Namely, for $X\in\rmodcatin{A^{G_1}}{G_2}$ and $U\in\Rep G_1$, we have
$$
XU=X\otimes_{A^{G_1}}(H_U\otimes A)^{G_1}.
$$
In order to complete the proof of Theorem~\ref{thm:main} it remains
to establish the following.

\begin{lemma}
The $(\Rep G_2)$-$(\Rep G_1)$-module categories $\DD_A$ and $\rmodcatin{A^{G_1}}{G_2}$ are equivalent.
\end{lemma}

\bp
By Proposition~\ref{prop:imprim} we have an equivalence of C$^*$-categories $E\colon\DD_A\to\rmodcatin{A^{G_1}}{G_2}$ given by $E(X)=X^{G_1}$. We want to enrich it to an equivalence of module categories. For this we have to define natural unitary isomorphisms
$$
\theta_{V,X,U}\colon V(E(X)U)=H_V\otimes X^{G_1}\otimes_{A^{G_1}}(H_U\otimes A)^{G_1}\to E(V(XU))=H_V\otimes (H_U\otimes X)^{G_1}
$$
in $\rmodcatin{A^{G_1}}{G_2}$ for $U\in\Rep G_1$, $X\in\DD_A$, $V\in\Rep G_2$. We define them by
$$
\theta_{V,X,U}(\xi\otimes x\otimes (\zeta\otimes a))=\xi\otimes(\zeta\otimes x a).
$$
It is clear that this is a $G_2$-equivariant morphism of right
$A^{G_1}$-modules. It is also easy to check that $\theta_{V,X,U}$ is
isometric. In order to check that such morphisms are unitary it
suffices to consider modules of the form $X=H_W\otimes A$ for
$W\in\Rep G_2$, since any object in $\DD_A$ is a subobject of such a
module by Corollary~\ref{cor:generation}. But for such modules the
statement is obvious. It is then straightforward to check that
$(E,\theta)$ is an equivalence of $(\Rep G_2)$-$(\Rep G_1)$-modules
categories. \ep

\subsection{Fiber functors on categories of bimodules}

In the previous sections we have developed an analogue of the bi-Hopf--Galois theory for categorically Morita equivalent compact quantum groups. We now turn to an analogue of the correspondence between fiber functors and Hopf--Galois objects.

\begin{definition}\label{def:MG}
For a compact quantum group $G$ and a finite dimensional simple
right $G$-C$^*$-algebra $B$, a $G$-\emph{Morita--Galois object} for
$B$ is a unital C$^*$-algebra $A$ together a free action
$\alpha\colon A\to A\otimes C(G)$, and a $G$-equivariant embedding
$B\hookrightarrow A$ such that there is a $G$-equivariant
isomorphism
$$
A^G\otimes\A\cong B\otimes\A
$$
of $A^{G}\otimes B$-$\A$-modules that maps $1\otimes 1\in A^G\otimes \A$ into an element of $B\otimes B\subset B\otimes\A$.
\end{definition}

Similarly to Proposition~\ref{prop:equivd}, existence of an
isomorphism as in the above definition can be reformulated as
follows.

\begin{lemma} \label{lem:equivd-1side}
Assume we are given a right action $\alpha\colon A\to A\otimes C(G)$ of a compact quantum group $G$ on a unital C$^*$-algebra $A$ and a finite dimensional invariant unital C$^*$-subalgebra $B\subset A$. Then a $G$-equivariant isomorphism $A^{G}\otimes\A\cong B\otimes\A$ of $A^G\otimes B$-$\A$-modules, mapping $1\otimes 1$ into an element of $B\otimes B$, exists if and only if the following conditions hold:
\begin{itemize}
\item[(a)] the fixed point algebra $A^{G}$ is finite dimensional;
\item[(b)] there exist a faithful $G$-invariant state $\psi_B$ on $B$ and a faithful state $\psi_{A^G}$ on $A^{G}$ such that if $m^*_B(1)=x^i\otimes x_i$ with respect to $\psi_B$ and $m^*_{A^G}(1)=y^j\otimes y_j$ with respect to $\psi_{A^G}$, then
$$
x^i y^j x_i\otimes y_j=\lambda 1\otimes1\ \ \text{and}\ \ y^j x^i y_j\otimes x_i=\lambda1\otimes1
$$
for a nonzero scalar $\lambda$.
\end{itemize}
Furthermore, if these conditions are satisfied, then
\begin{itemize}
\item[(i)] the map $A^{G}\otimes\A\to B\otimes\A,\ \ a\otimes c\mapsto x^i\otimes a x_i c$, is a $G$-equivariant isomorphism of $A^{G}\otimes B$-$\A$-modules, with the inverse given by $e\otimes f\mapsto \lambda^{-1} y^j\otimes e y_j f$;
\item[(ii)] as the state $\psi_B$ we can take the canonical $G$-invariant state $\varphi_B$ on $B$, in which case $\lambda=\dim_q B$;
\item[(iii)] the relative commutants $(A^{G})'\cap B$ and $B'\cap A^{G}$ are trivial; in particular, $B$ is a simple $G$-C$^*$-algebra.
\end{itemize}
\end{lemma}

Let us also note that we have the following analogue of Lemma~\ref{lem:iso}, with identical proof.

\begin{lemma}\label{lem:exchange}
Let $Y \in \bimodcatin{B}{G}$ and $Y'$ be a $G$-equivariant
$A^{G}$-$A$-correspondence. Then we have a $G$-equivariant
isomorphism
$$
Y \otimes Y' \cong Y'^{G} \otimes (Y \otimes_B A)
$$
of $A^G \otimes B$-$A$-correspondences. In particular, for any $V\in
\Rep G$, we have a $G$-equivariant isomorphism
$$
B\otimes H_V\otimes A\cong (H_V\otimes A)^{G}\otimes A
$$
of $B\otimes A^{G}$-$A$-correspondences.
\end{lemma}

We then have the following result.

\begin{theorem}
\label{thm:one-sided-Morita-Galois}
For any reduced compact quantum group $G$ and any finite dimensional simple
right $G$-C$^*$-algebra~$B$, there is a one-to-one correspondence
between the isomorphism classes of $G$-Morita--Galois objects for
$B$ and the isomorphism classes of unitary fiber functors
$\bimodcatin{B}{G}\to\Hilbf$.
\end{theorem}

\bp Assume we are given a $G$-Morita--Galois object $A$ for $B$ as
in Definition~\ref{def:MG}. We define a functor
$F\colon\bimodcatin{B}{G}\to\Hilbf$ by
$$
F(X)=(\text{the space of }G\text{-invariant }B\text{-central vectors
in }X\otimes_BA).
$$
We will see later that the space $F(X)$ is finite dimensional. The
Hilbert space structure is defined as follows. The space
$X\otimes_BA$ is a right Hilbert $A$-module. If $\xi,\zeta\in F(X)$,
then $\langle \xi,\zeta\rangle_A \in B'\cap A^G=\C1$, so we can define
a scalar product by $(\zeta,\xi)1=\langle \xi,\zeta\rangle_A$.

Next, we define a tensor structure on $F$ by
$$
F_2\colon F(X)\otimes F(Y)\to F(X\otimes_BY),\ \ F_2((\xi\otimes
a)\otimes (\zeta\otimes c))=(\xi\otimes\zeta)\otimes ca.
$$
In order to check that $F_2$ is unitary it suffices to consider
bimodules of the form $B\otimes H_V\otimes B$ for $V\in\Rep G$,
since any $X\in\bimodcatin{B}{G}$ embeds isometrically into
$B\otimes X\otimes B$. By Lemma~\ref{lem:exchange} we have a
$G$-equivariant isomorphism of $B$-$A$-modules
\begin{equation} \label{eq:GM1}
B\otimes H_V\otimes A\cong (H_V\otimes A)^G\otimes A,
\end{equation}
so we can define a linear isomorphism
\begin{equation}\label{eq:T}
T_V\colon (H_V\otimes A)^G\to F(B\otimes H_V\otimes B)\subset
B\otimes H_V\otimes A, \ \ \xi\otimes a\mapsto (\dim_q
B)^{-1/2}x^i\otimes \xi\otimes a x_i,
\end{equation}
which in particular shows that $F(X)$ is indeed finite dimensional
for any $X\in\bimodcatin{B}{G}$. Note that the $A$-valued inner product on
$B\otimes H_V\otimes A$ is given by
$$
\langle b_1\otimes\xi_1\otimes a_1,b_2\otimes\xi_2\otimes
a_2\rangle_A =\varphi_B(b_1^*b_2)(\xi_2,\xi_1)a_1^*a_2.
$$
Therefore if we define a scalar product on $(H_V\otimes A)^G$ in the
standard way,
$$
(\xi_1\otimes a_1,\xi_2\otimes a_2)
= \varphi_{A^G}(\langle \xi_2\otimes a_2,\xi_1\otimes a_1\rangle_A)
= (\xi_1,\xi_2) \varphi_{A^G}(a_2^*a_1),
$$
then $T_V$ becomes unitary, since
\begin{equation} \label{eq:GM3}
\varphi_B(x^{i*}x^j)x_i^*ax_j=(\dim_q B)\varphi_{A^G}(a)1\ \
\text{for all}\ \ a\in A^G.
\end{equation}
Thus, in order to show that the maps $F_2$ are unitary it suffices
to check that, for all $U$ and $V$,
\begin{gather*}
T^{-1}_{H_U\otimes B\otimes H_V}F_2(T_U\otimes T_V)\colon
(H_U\otimes A)^G\otimes (H_V\otimes A)^G\to (H_U\otimes B\otimes
H_V\otimes A)^G,\\
(\xi\otimes a)\otimes(\zeta\otimes c)\mapsto (\dim_q
B)^{-1/2}\xi\otimes x^i\otimes \zeta\otimes c x_i a,
\end{gather*}
is unitary. It is clear from \eqref{eq:GM3} that this map is an
isometry, so to prove that it is a unitary isomorphism it is enough
to compare the dimensions of both sides. Using \eqref{eq:GM1} again,
we get isomorphisms
$$
(H_U\otimes B\otimes H_V\otimes A)^G\cong (H_U\otimes(H_V\otimes
A)^G\otimes A)^G\cong (H_U\otimes A)^G\otimes (H_V\otimes A)^G,
$$
which completes the proof of unitarity of $F_2$. We have thus proved
that $(F,F_2)$ is a unitary tensor functor.

\smallskip

Let us show next that the spectral functor $(\Rep G)^{\otimes\op}\to
\bimodcat{A^G}$ defined by the action of $G$ on $A$ can be
reconstructed from $F$. Consider the dual C$^*$-Frobenius algebra
$B\otimes B\in\bimodcatin{B}{G}$ with product
$$
(\dim_q B)^{1/2}\iota\otimes\varphi_B\otimes\iota\colon B\otimes
B\otimes B=(B\otimes B)\otimes_B(B\otimes B)\to B\otimes B
$$
and unit $(\dim_q B)^{-1/2}m_B^*\colon B\to B\otimes B$. By applying
$F$ we get a C$^*$-Frobenius object in $\Hilbf$, that is, by
Lemma~\ref{lem:FrobeniusVScstar}, a Frobenius C$^*$-algebra. It is
easy to see that the unitary
$$
T_\un\colon A^G\to F(B\otimes B)\subset B\otimes A,\ \ a\mapsto
(\dim_q B)^{-1/2}x^i\otimes ax_i,
$$
is an isomorphism of $((A^G)^\op,\varphi_{A^G})$ with this Frobenius
C$^*$-algebra. Similarly, any $B$-bimodule $B\otimes H_U\otimes B$
is a $(B\otimes B)$-bimodule in $\bimodcatin{B}{G}$, so $F(B\otimes
H_U\otimes B)$ becomes an $F(B\otimes B)$-bimodule, and using the
isomorphisms $T_U\colon (H_U\otimes A)^G\to F(B\otimes H_U\otimes
B)$ and $T_\un\colon (A^G)^\op\to F(B\otimes B)$ we recover the
$A^G$-bimodule structure on $(H_U\otimes A)^G$. Finally, one can
also easily check that the tensor structure of the spectral functor
can be recovered from that $F_2$ and the maps
$F(\iota\otimes\iota\otimes\varphi_B\otimes\iota\otimes\iota)\colon
F(B\otimes H_V\otimes B\otimes H_U\otimes B)\to F(B\otimes
H_V\otimes H_U\otimes B)$.

\smallskip

Assume now that we have another $G$-Morita--Galois object $\tilde A$
for $B$ defining an isomorphic fiber functor~$\tilde F$. Let
$\eta\colon F\to\tilde F$ be such a unitary monoidal natural
isomorphism. It follows from the above discussion that we then get
an isomorphism $A^G\cong \tilde A^G$ intertwining the spectral
functors $(\Rep G)^{\otimes\op}\to\bimodcat{A^G}$ and $(\Rep
G)^{\otimes\op}\to\bimodcat{\tilde A^G}$. Hence we get a
$G$-equivariant isomorphism $\theta\colon A\to\tilde A$. We claim
that $\theta$ is the identity map on $B$, so that $\theta$ is an
isomorphism of Morita--Galois objects for $B$.

In view of the way we obtained an isomorphism of the spectral
functors, we have commutative diagrams
$$
\begin{tikzcd}
(H_U\otimes A)^G\ar[rr, "\iota\otimes\theta"]\ar[d] & & (H_U\otimes\tilde A)^G\ar[d]\\
F(B\otimes H_U\otimes B)\ar[rr, "\eta_{B\otimes H_U\otimes B}"']&
&\tilde F(B\otimes H_U\otimes B),
\end{tikzcd}
$$
where the vertical arrows are the maps $T$ defined by~\eqref{eq:T}.
In other words, for any $\xi\otimes a\in (H_U\otimes A)^G$ we have
$$
\eta_{B\otimes H_U\otimes B}(x^i\otimes \xi\otimes
ax_i)=x^i\otimes\xi\otimes\theta(a)x_i.
$$
Using that $\eta F_2=\tilde F_2(\eta\otimes\eta)$ we then get that
for any $\zeta\otimes c\in (H_V\otimes A)^G$ we have
$$
x^i\otimes \xi\otimes x^j\otimes\zeta\otimes
\theta(c)x_j\theta(a)x_i=x^i\otimes \xi\otimes
x^j\otimes\zeta\otimes \theta(c x_j a)x_i
$$
in $B\otimes H_U\otimes B\otimes H_V\otimes \tilde A$. In the
simplest case $U=V=\un$ this gives
$$
x^i\otimes x^j\otimes x_j x_i=x^i\otimes x_j\otimes\theta(x_j)x_i,
$$
and applying $\varphi_B$ to the first leg we obtain $x^j\otimes
x_j=x^j\otimes \theta(x_j)$. Hence $\theta(x_j)=x_j$, so $\theta$ is
the identity map on $B$.

\smallskip

It is also clear that isomorphic Morita--Galois objects define
isomorphic fiber functors. It remains to show that any unitary fiber
functor is defined by a Morita--Galois object. Assume we are given
such a functor
$$
E \colon \bimodcatin{B}{G} \to \Hilb_f.
$$
By Woronowicz's Tannaka--Krein duality it
defines a compact quantum group $G_1$. Then $\rmodcatin{B}{G}$
becomes an invertible $(\Rep G)$-$(\Rep G_1)$-module category with
generator $B$ and we can consider the corresponding
$G_1$-$G$-Morita--Galois object $A$.

We claim that the canonical fiber functor $\Rep G_1\to\Hilbf$ is
isomorphic to the composition of the spectral functor $E_A\colon
\Rep G_1\to\bimodcatin{B}{G}$ for $G_1 \curvearrowright A$, with the
functor $F_A\colon\bimodcatin{B}{G}\to\Hilbf$ corresponding to $A$
as defined at the beginning of the proof. We thus have to define a
natural isomorphism
$$
H_U\to F((H_U\otimes A)^{G_1})\subset (H_U\otimes A)^{G_1}\otimes_B
A.
$$
As $(H_U\otimes A)^{G_1}\otimes_B A\cong H_U\otimes A$ by
Proposition~\ref{prop:imprim}, it is straightforward to check that
$H_U\ni\xi\mapsto \xi\otimes1\in H_U\otimes A$ is the required
isomorphism.

In other words, we have proved that the spectral functor associated
with the action of $G_1$ on $A$ gives an autoequivalence $E_A$ of
$\Rep G_1=\bimodcatin{B}{G}$ such that $F_AE_A\cong E$. Now, in
order to complete the proof of the theorem it would be enough to
show that $E_A$ is isomorphic to the identity functor. It is indeed
possible to do so, but let us instead finish the proof by giving a
more formal argument, as follows.

Suppose that $\tilde E\colon\bimodcatin{B}{G}\to\Hilb_f$ is a
unitary fiber functor, and let~$\tilde G_1$ be the corresponding
compact quantum group. Consider the bi-Hopf--Galois object $\tilde
A$ corresponding to the pair of functors $\tilde E$ and $E$. In
other words, $\tilde A$ is the Morita--Galois object defined by the
category $\Rep G_1$, considered as a $(\Rep G_1)$-$(\Rep \tilde
G_1)$-module category, and by the object $\un\in\Rep G_1$. Then by
Proposition~\ref{prop:cotensor} below, the cotensor product
$\tilde\A \boxvoid_{G_1} \A$ is the regular subalgebra of a $\tilde
G_1$-$G$-Morita--Galois object~$C$, and~$C^{\tilde G_1}$ is
canonically isomorphic to $A^{G_1}=B$.

By definition we have $\Rep \tilde{G}_1 = \bimodcatin{B}{G}$. Under
this identification, the spectral functor associated with
$\tilde{G}_1 \curvearrowright \tilde{A}$, which corresponds to the
monoidal equivalence $\Rep \tilde{G}_1 \to \Rep G_1$, is just the
identity functor on $\bimodcatin{B}{G}$. Similarly, the spectral
functor associated with $\tilde G_1 \curvearrowright C$ can be
regarded as an autoequivalence~$E_C$ of $\bimodcatin{B}{G}$, which
is naturally unitarily monoidally isomorphic to $E_A$ by
associativity of the cotensor product operation. We thus get
$F_CE_A\cong F_CE_C,$ but the latter is isomorphic to $\tilde E$ by
the same observation as for $A$ and $E$ above. In particular, if we
started with $\tilde E=EE_A$, we would get $F_C\cong E$.
 \ep

\section{Categorical Morita equivalence} \label{sec:relten}

\subsection{Weak monoidal Morita equivalence and tensor product of bimodule categories}

Recall that two rigid C$^*$-tensor categories $\CC_1$ and $\CC_2$
are called \emph{unitarily weakly monoidally Morita
equivalent}~\cite{MR1966524} if there exists a rigid
C$^*$-$2$-category $\CC$ with the set $\{1,2\}$ of $0$-cells such
that $\CC_{11}$ and $\CC_{22}$ are unitarily monoidally equivalent
to $\CC_1$ and $\CC_2$ respectively, and $\CC_{12}\ne0$, or in other
words, if there exists an invertible $\CC_1$-$\CC_2$-bimodule
category. Using Frobenius algebras it is shown in \cite{MR1966524}
that this is indeed an equivalence relation. In the fusion category
case a more transparent proof is obtained using relative tensor
product of bimodule categories. We now want to make sense of this in
our infinite C$^*$-setting.

In fact, we will show a bit more. By passing to equivalent
categories we may assume that $\CC_{11}=\CC_1$ and $\CC_{22}=\CC_2$.
Assume also that $\CC_2$ is unitarily weakly monoidally Morita
equivalent to a third rigid C$^*$-tensor category $\CC_3$, and let
$(\cC_{ij})_{i,j=2,3}$ be the corresponding rigid
C$^*$-$2$-category. Let us show then that the two $2$-categories
$(\cC_{ij})_{i,j=1}^2$ and $(\cC_{ij})_{i,j=2}^3$ can be `combined'
into a C$^*$-$2$-category with $0$-cells $\{1, 2, 3\}$. We thus need
to define $\cC_{13}$, $\cC_{31}$ as bimodule categories over $\cC_1$
and $\cC_3$, and define the horizontal compositions $\cC_{1 3}
\times \cC_{3 1} \to \cC_1$, $\cC_{13} \times \cC_{32} \to
\cC_{12}$, etc. The idea is simple: using the duality morphisms we
can express everything in terms of the categories that we already
have.

Thus, we define $\cC_{13}$ as the idempotent completion of the
category with objects $X Y$ for $X \in \cC_{12}$ and $Y \in
\cC_{23}$, with respect to the morphism sets
$$
\cC_{13}(X Y, X' Y') = \cC_2(\bar{X}' X, Y' \bar{Y}).
$$
In the following exposition, let us denote by $S_0$ the
representative in $\cC_2(\bar{X}' X, Y' \bar{Y})$ of a morphism
$S\colon X Y \to X' Y'$. The composition of morphisms in $\cC_{13}$
is then defined by
$$
(S \circ T)_0 = (\iota \otimes R_{Y'}^* \otimes \iota) (S_0 \otimes T_0) (\iota \otimes \bar{R}_X \otimes \iota)
$$
for $T \in \cC_{13}(X Y, X' Y') $ and $S \in \cC_{13}(X' Y', X''
Y'')$, so that $(R_X^* \otimes \iota_{Y\bar Y}) (\iota_{\bar XX}
\otimes \bar{R}_Y)$ represents the identity morphism of $X Y$.
Moreover, $(X, Y) \mapsto X Y$ is a bifunctor: for $S \in
\cC_{12}(X, X')$ and $T \in \cC_{23}(Y, Y')$, the morphism $S
\otimes T \colon X Y \to X' Y'$ is represented by $(T_0 \otimes
\iota) \bar{R}_Y R_{X'}^* (\iota \otimes S_0)$.

The left $\cC_1$-module category structure is defined by $U (X Y) =
(U X) Y$ at the level of objects, and by
$$
(S \otimes T)_0 = T_0 ((R_{U'}^* (\iota \otimes S)) \otimes \iota)
$$
for $S \in \cC_1(U, U')$ and $T \in \cC_{13}(X Y, X' Y')$, at the
level of morphisms. The right $\cC_3$-module category structure on
$\cC_{13}$ is defined in a similar way. The $\cC_3$-$\cC_1$-module
category $\cC_{3 1}$ is also defined in an analogous way as the
idempotent completion of the category of objects $Z W$ for $Z \in
\cC_{32}$ and $W \in \cC_{21}$, with morphism sets
$$
\cC_{31}(Z W, Z' W') = \cC_2(W \bar{W}', \bar{Z} Z').
$$

The horizontal composition $\cC_{1 3} \times \cC_{3 1} \to \cC_1$ is
given by $(X Y) (Z W) = X (Y Z) W$ at the level of objects, and at
the level of morphisms $S \otimes T \in \cC_1(X Y Z W, X' Y' Z'
W')$, for $S \in \cC_{13}(X Y, X' Y')$ and $T \in \cC_{31}(W Z, W'
Z')$, is given by
$$
((((\iota \otimes R_Y^*)(\iota_{X'} \otimes S_0 \otimes \iota)) \otimes ((\bar{R}_Z^* \otimes \iota)(\iota \otimes T_0 \otimes \iota_{W'}))) (\bar{R}_{X'} \otimes \iota \otimes R_{W'}).
$$
The horizontal composition $\cC_{3 1} \times \cC_{1 3} \to \cC_3$ is defined in a similar way. Next let us describe $\cC_{1 3} \times \cC_{3 2} \to \cC_{1 2}$. At the level of objects, it is given by $(X Y) Z = X (Y Z)$ for $X \in \cC_{1 2}$, $Y \in \cC_{2 3}$, and $Z \in \cC_{3 2}$. At the level of morphisms, $S \otimes T \in \cC_{1 2}(X Y Z, X' Y' Z')$ for $S \in \cC_{1 3}(X Y, X' Y')$ and $T \in \cC_{3 2}(Z, Z')$ is given by
$$
(\iota_{X' Y' Z'} \otimes R_{Y Z}^*) (\iota_{X' Y'} \otimes T \otimes \iota_{\bar{Z} \bar{Y} Y Z}) (\iota_{X' Y'} \otimes \bar{R}_Z \otimes \iota_{\bar{Y} Y Z}) (\iota_{X'} \otimes S_0 \otimes \iota_{Y Z}) (\bar{R}_{X'} \otimes \iota_{XYZ}).
$$
The remaining horizontal compositions are defined similarly.

\begin{lemma}
The category $\cC_{1 3}$ is a C$^*$-category with the norm $\norm{S} = \norm{S \otimes \iota_{\bar{Y}}}_{\cC_{1 2}(X Y \bar{Y}, X' Y' \bar{Y})}$ and the involution $(S^*)_0 = (R_{\bar{X}' X} \otimes \iota) (\iota \otimes S_0^* \otimes \iota) (\iota \otimes \bar{R}_{Y' \bar{Y}}) \in \cC_2(\bar{X} X', Y \bar{Y}')$ for $S \in \cC_{1 3}(X Y, X' Y')$.
\end{lemma}

\bp Take any nonzero object $Z\in\cC_{23}$. It is easy to check that
we can define a faithful $*$-preserving functor
$F_Z\colon\cC_{13}\to\cC_{12}$ by letting $F_Z(XY)=(XY)\bar Z$ on
objects and $F_Z(S)=S\otimes\iota_{\bar Z}$ on morphisms. It follows
that the $*$-operation in the formulation of the proposition is
indeed an involution and that the C$^*$-norm on morphisms in
$\cC_{12}$ defines a C$^*$-norm on morphisms in $\cC_{13}$. The
latter norm is independent of the choice of $Z$, since any other
object $Z'\in\cC_{23}$ embeds into $Z(\bar Z Z')$. \ep

In a similar way we check that $\cC_{31}$ is a C$^*$-category. A
straightforward verification shows then that $(\cC_{i j})_{i,j=1}^3$
is a rigid C$^*$-$2$-category.

\smallskip

We denote the invertible $\CC_1$-$\CC_3$-module category $\CC_{13}$
by $\CC_{12}\boxtimes_{\cC_{2}}\CC_{23}$. Note that using
representatives $(V_j)_j$ of the isomorphism classes of simple
objects in $\cC_2$, we can write
\begin{equation}
\label{eq:C13-mor-Vj-barVj-pres} \cC_{13}(X Y, X' Y') = \bigoplus_j
\cC_2(\bar{X}' X, V_j) \otimes \cC_2(V_j, Y' \bar{Y}) \cong
\bigoplus_j \cC_{1 2}(X, X' V_j) \otimes \cC_{2 3}(Y, \bar{V}_j Y').
\end{equation}

\begin{remark}
\label{rem:2-cat-compos-monad-pic} Consider the Deligne tensor
product $\cC_{12} \boxtimes \cC_{23}$, which is the category with
objects $X \boxtimes Y$ and morphisms
$$
\Mor_{\cC_{12} \boxtimes \cC_{23}}(X \boxtimes Y, X' \boxtimes Y') =
\cC_{12}(X, X') \otimes \cC_{23}(Y, Y').
$$
The functor $T(X \boxtimes Y) = \bigoplus_j X V_j \boxtimes
\bar{V}_j Y$ is an endofunctor of the ind-category of $\cC_{12}
\boxtimes \cC_{23}$. Decomposition of the tensor products $V_j
V_{j'}$ into simple objects induces the structure of a \emph{monad} on $T$,
that is, a natural transformation $T^2 \to T$ (together with $\Id
\to T$). Formula~\eqref{eq:C13-mor-Vj-barVj-pres} shows that the
morphism sets in $\cC_{1 3}$ are given by
$$
\cC_{1 3}(X Y, X' Y') \cong \Mor_{\cC_{12} \boxtimes \cC_{23}}(X
\boxtimes Y, T(X' \boxtimes Y'))
$$
for $X, X' \in \cC_{12}$ and $Y, Y' \in \cC_{23}$. The right hand
side of the above can be regarded as the set of $T$-module morphisms
between the free $T$-modules $T(X \boxtimes Y)$ and $T(X' \boxtimes
Y')$. Thus, $\cC_{1 3}$ can be interpreted as the category of
$T$-modules in $\cC_{12} \boxtimes \cC_{23}$, and $X Y$ is
represented by $T(X \boxtimes Y)$.
\end{remark}

\begin{remark}
\label{rem:2-cat-compos-Q-sys-pic}
By~\cite{MR1966524}*{Proposition~4.5}, or by
Theorem~\ref{thm:inv-bimod-cat-is-semisimple-proper-mod-cat}, we may
assume that $\cC_{12}=\modQin{\cC_1}$ and $\cC_2 =
\bimodcatin{Q}{\cC_1}$ for a standard $Q$-system $(Q, m,v)$ in
$\cC_1$, and $\cC_{2 3}=\rmodcatin{Q'}{\cC_2}$ and $\cC_3 =
\bimodcatin{Q'}{\cC_2}$ for a standard $Q$-system $(Q', m', v')$ in
$\cC_2$. Then, denoting by $P_{Q',Q'}\colon Q' \otimes Q' \to Q'
\otimes_Q Q'$ the structure morphism of the tensor product over $Q$,
the morphisms $\tilde{m} = m' P_{Q',Q'}$ and $\tilde{v} = v' v$
define the structure of a standard $Q$-system on $Q'$ as an object
in $\cC_1$. We claim that $\cC_{1 3}$ is equivalent to
$\rmodcatin{Q'}{\cC_1}$ as a $\cC_1$-$\cC_3$-module category in such a way that
$X Y$ corresponds to $X \otimes_Q Y$ for $X \in \modQin{\cC_1}$ and
$Y \in \rmodcatin{Q'}{\bimodcat{Q}}$ (note
that the category $\bimodcatin{Q'}{\cC_1}$ can be regarded as $\cC_3
= \bimodcatin{Q'}{\bimodcat{Q}}$, since any $Q'$-module is also a
$Q$-module by the inclusion $v'\colon Q \to Q'$).

Indeed, $X \otimes_Q Y$ inherits the structure of a right
$Q'$-module from $Y$, and by the Frobenius reciprocity we have
$$
\Mor_{\rmodcatin{Q'}{\cC_1}}(X \otimes_Q Y, X' \otimes_Q Y') \cong
\Mor_{\bimodcat{Q}}(\bar{X}' \otimes X, Y' \otimes_{Q'} \bar{Y}).
$$
This shows that the subcategory of $\rmodcatin{Q'}{\cC_1}$ generated
by the objects of the form $X \otimes_Q Y$ is equivalent
to~$\cC_{13}$. But this is the whole category
$\rmodcatin{Q'}{\cC_1}$, since any right $Q'$-module $X$ in $\cC_1$
is a submodule of $X\otimes_QQ'$.
\end{remark}

\subsection{Cotensor product of bi-Morita--Galois objects}

At the level of bi-Morita--Galois objects relative tensor product of
bimodule categories corresponds to cotensor product. In the Hopf
algebra setting this result has been already obtained in
\cite{MR2989520}, so we will only give a sketch of an alternative
argument in our C$^*$-setting. Note that this result does not need a
characterization of algebras arising from invertible bimodule
categories.

\begin{proposition}\label{prop:cotensor}
Let $G_1$, $G_2$ and $G_3$ be compact quantum groups, $A$ be a
$G_1$-$G_2$-Morita--Galois object and~$B$ be a
$G_2$-$G_3$-Morita--Galois object. Consider the bimodule categories
$\DD_A$ and $\DD_B$, and let $X\in\DD_A$ and $Y\in\DD_B$ be the
generators corresponding to $A$ and $B$, respectively. Then the
$G_1$-$G_3$-Morita--Galois object corresponding to the invertible
bimodule category $\DD_B\boxtimes_{\Rep G_2}\DD_A$ and its generator
$YX$ is the completion of $\A\boxvoid_{G_2}\BB$.
\end{proposition}

\begin{proof}
We write $\cC_n$ for $\Rep G_n$ ($n=1,2,3$), $\cC_{32}$ for $\DD_B$
and $\cC_{21}$ for $\DD_A$. Choose representatives $(U_i)_i$,
$(V_j)_j$ and $(W_k)_k$ of the isomorphism classes of irreducible
representations of $G_1$, $G_2$, and $G_3$ respectively. The regular
subalgebra of the $G_1$-$G_3$-C$^*$-algebra corresponding to $Y X\in
\cC_{31}=\DD_B\boxtimes_{\Rep G_2}\DD_A$ is given by
$$
\bigoplus_{\substack{i\in \Irr (G_1)\\ k\in \Irr (G_3)}} \bar{H}_i
\otimes \cC_{31}(Y X, W_k Y X U_i) \otimes \bar H_k.
$$
Similarly to~\eqref{eq:C13-mor-Vj-barVj-pres}, we have
\begin{multline*}
\cC_{31}(Y X, {W}_k Y X U_i)\cong\cC_2(X\bar U_k\bar X,\bar YW_kY)\\
\cong \bigoplus_{j \in \Irr (G_2)} \cC_2(X \bar{U}_i \bar{X}, V_j)
\otimes \cC_2(V_j, \bar{Y} {W}_k Y) \cong \bigoplus_{j \in \Irr
(G_2)} \cC_{21}(X, {V}_j X U_i) \otimes \cC_{32}(Y, {W}_k Y\bar
V_j).
\end{multline*}
From this we see that the regular subalgebra is isomorphic to $\A
\boxvoid_{G_2} \BB$ as a left $\C[G_1]$-comodule and a right
$\C[G_3]$-comodule. It is also not difficult to compare the products
and involutions on the two algebras.
\end{proof}

\subsection{Categorical Morita equivalence and Brauer--Picard group}

Similarly to \cite{MR2362670} we give the following definition.

\begin{definition}
Two compact quantum groups $G_1$ and $G_2$ are called
\emph{categorically Morita equivalent} if there is an invertible
$(\Rep G_2)$-$(\Rep G_1)$-module category.
\end{definition}

By Theorem~\ref{thm:main}, two compact quantum groups $G_1$ and
$G_2$ are categorically Morita equivalent if and only if there
exists a $G_1$-$G_2$-Morita--Galois object.

\smallskip

An invertible bimodule category implementing categorical Morita equivalence of $G_1$ and $G_2$ is by no means unique. This leads to a notion of the Brauer--Picard group~\cite{MR2677836}. Namely, in our analytic setting, by the \emph{Brauer--Picard group} of a rigid C$^*$-tensor category $\CC$ we mean the set $\BrPic(\CC)$ of equivalence classes of invertible $\CC$-bimodule categories, with the group law defined by the relative tensor product $\boxtimes_\CC$. For $\CC=\Rep G$, we can equivalently define $\BrPic(\Rep G)$ as the set of equivariant Morita equivalence classes of $G$-$G$-Morita--Galois objects, with the group law defined by the cotensor product over $G$. We will discuss these notions for compact quantum groups in detail elsewhere, confining ourselves for the moment to a few examples and remarks.

\begin{example}
Any finite quantum group $G$ is categorically Morita equivalent to
its dual $\hat G$. This follows by considering a depth $2$ subfactor
$N\subset N\rtimes G$ and was already observed by
M{\"u}ger~\cite{MR1966524}*{Corollary~6.16}, but let us show this
using Morita--Galois objects.

Consider the C$^*$-algebra $A=C(G)\rtimes G$, where $G$ acts on
$C(G)$ by right translations. The action of~$G$ on~$C(G)$ by left
translations extends in the obvious way to an action on $A$, while $\hat{G}$ also acts on $A$ by the dual action.
These two
actions commute and we claim that $A$ is a $G$-$\hat
G$-Morita--Galois object.

Since the action $G$ on $C(G)$ by left translations is free, the
action of $G$ on $A$ is also free by Proposition~\ref{prop:freesubalg}.
For similar reasons the action of $\hat G$ is free. Next, let
$u^s_{ij}$, $s\in\Irr(G)$, $i,j=1,\dots,d_s$,  be matrix
coefficients of irreducible unitary representations of $G$. The dual
basis with respect to the Haar state is given by $d_s u^{s*}_{ij}$.
For any $\omega\in C(\hat G) = C^* G \subset A$, we have
$$
\sum_{s,i,j}d_s u^{s*}_{ij}\omega u^s_{ij}=\sum_{s,i,j,k}d_s(u^s_{kj},\omega_{(1)})u^{s*}_{ij}u^s_{ik}\omega_{(2)}=\sum_{s,j}d_s(u^s_{jj},\omega_{(1)})\omega_{(2)}.
$$
Up to normalization, the Haar state $h_{\hat G}$ on $C(\hat
G)\cong\bigoplus_s \Mat_{d_s}(\C)$ is given by $\bigoplus_s d_s\Tr$.
Hence, up to a scalar factor, the above expression equals
$$
h_{\hat G}(\omega_{(1)})\omega_{(2)}=h_{\hat G}(\omega)1.
$$
Therefore the second identity in \eqref{eq:Mkey} is satisfied for
$G_1=G$ and $G_2=\hat G$, as $A^G=C(\hat G)$ and $A^{\hat G}=C(G)$.
Since the roles of $G$ and $\hat G$ are symmetric, the first
identity there is satisfied as well. Hence $A$ is indeed
a $G$-$\hat G$-Morita--Galois object. Note also that the canonical
invariant state on $A \cong M_{\dim C(G)}(\C)$ is the unique
tracial state.
\end{example}

\begin{example}
Assume $G$ is a genuine compact group and $\pi\colon
G\to\operatorname{PU}(H)$ is a projective unitary representation of
$G$ on a finite dimensional Hilbert space $H$. Consider the algebra
$C(G)\otimes B(H)$ with two commuting actions of $G$: one action is
given by left translations of $G$ on $C(G)$, the other by the tensor
product of the action by right translations on $C(G)$ and by
$\Ad\pi$ on $B(H)$. These actions are free by
Proposition~\ref{prop:freesubalg} and both fixed point algebras are
isomorphic to $B(H)$. Taking the unique tracial states on these
algebras it is easy to check that identities \eqref{eq:Mkey} are
satisfied. Therefore $A$ is a $G$-$G$-Morita--Galois object.

In particular, the categories $\Rep G$ and $\bimodcatin{B(H)}{G}$ are
unitarily monoidally equivalent. Modulo unitarity, this, in fact,
follows already from~\cite{MR0498794} (see also~\cite{MR1614178}*{Corollary~3.2}), since $B(H)$
is an Azumaya algebra in the symmetric monoidal category $\Rep G$.

This simple example has the following consequence: if $G$ is a genuine compact connected group, then any compact quantum group categorically Morita equivalent to $G$ is monoidally equivalent to $G$. Indeed, if $G'$ is such a compact quantum group, then $\Rep G'$ is unitarily monoidally equivalent to $\bimodcatin{B}{G}$ for some simple $G$-C$^*$-algebra $B$. Since $G$ is connected, $B$ must be a full matrix algebra $B(H)$, and the claim follows.
\end{example}

\appendix
\section{\texorpdfstring{$Q$}{Q}-systems and proper module categories}

The goal of this appendix is to prove the following correspondence
between $Q$-systems and proper module categories.

\begin{theorem}
\label{thm:proper-mod-equiv-to-Q-sys-mod} Let $\cC$ be a rigid
C$^*$-tensor category with simple unit, and let $\cD$ be a nonzero
indecomposable semisimple proper right $\cC$-module category. Then
there is an irreducible $Q$-system $A$ in $\cC$ such that $\cD$ is
unitarily equivalent to $\lmodcatin{A}{\cC}$ as a right $\cC$-module
category.
\end{theorem}

This is an adaptation to the infinite C$^*$-setting of a result of
Ostrik~\cite{MR1976459}*{Theorem~3.1}. It is certainly known to
experts, see, e.g.,~\cite{arXiv:1511.07982}*{Section~3},
but the precise details seem to be somewhat elusive in the
literature. The main point is to show that, if $\cD$ is a proper
right module category over $\cC$, the unitary structure induces the
structure of a $Q$-system on the internal endomorphism object
$\inEnd(X)$ for any $X$, cf.~\cite{MR2909758}*{p.~625}. This would
imply that $\cD$ is a part of a rigid C$^*$-bicategory which has
$\cC$ in one of its diagonal corners.

\smallskip

Fix representatives $(U_i)_i$ and $(X_a)_a$ of the isomorphism
classes of simple objects in $\cC$ and $\cD$ respectively. For any
$X \in \cD$, we always consider $\cD(X_a, X)$ and $\cD(X, X_a)$ as
Hilbert spaces equipped with the scalar products
\begin{align*}
(S, T)_{\cD(X_a, X)} \iota_a &= T^* S,&
(S, T)_{\cD(X, X_a)} \iota_a &= S T^*.
\end{align*}
More generally, for $X, Y \in \cD$, we consider $\cD(X, Y)$ as a
Hilbert space via the identification
$$
\cD(X, Y) \cong \bigoplus_a \cD(X, X_a) \otimes \cD(X_a, Y).
$$
This way the functor $Y \mapsto \cD(X, Y)$ is a C$^*$-functor from
$\cD$ to $\Hilb_f$ for any $X$.

The dual module category of $\cD$ is given by the C$^*$-category of
right $\cC$-module functors ${\Hom}_{\cC}(\cD, \cC)$. We have
canonical pairings ${\Hom}_{\cC}(\cD, \cC) \times \cD \to \cC$ and
$$
\cD \times {\Hom}_{\cC}(\cD, \cC) \to {\End}_{\cC}(\cD), \quad (X,
F) \mapsto (Y \mapsto X F (Y)).
$$
Fix a simple object $X$ in $\DD$. We define a `dual' of $X \in \cD$
as an object in ${\Hom}_\cC(\cD, \cC)$ by
$$
\bar{X} Y = \bigoplus_i \cD(X, Y \bar{U}_i) \otimes U_i,
$$
where we write $\bar XY$ instead of $\bar X(Y)$ and make use of a
unitary bifunctor $\Hilb_f \times \cC \to \cC$ characterized by
$$\cC(H \otimes U, K \otimes V) \cong B(H, K) \otimes \cC(U, V).$$
The object $\bar XY$ is well-defined by the properness assumption on
$\cD$. The functor $\bar X$ is adjoint to the functor $\CC\to\DD$,
$U\mapsto XU$, via the natural isomorphisms
\begin{equation}\label{eq:barX-as-internal-hom}
\begin{split}
\theta_{X,Y,U} \colon \cC(U, \bar{X}Y) = \bigoplus_i \cD(X, Y \bar{U}_i) \otimes \cC(U, U_i) \to \cD(X U, Y),\\
\cD(X, Y \bar{U}_i) \otimes \cC(U, U_i)\ni S \otimes T \mapsto d_i^{1/2} (\iota \otimes R_i^*) (S \otimes T),
\end{split}
\end{equation}
for $U \in \cC$. In other words, $\bar XY$ is the internal Hom
object $\inHom(X,Y)$. Here and below we identify $U$ with $\C\otimes
U$, so that when $H$ is a finite dimensional Hilbert space the space of morphisms $U\to H\otimes V$ equals $B(\C,H)\otimes\CC(U,V)=H\otimes\CC(U,V)$. The
natural isomorphisms $(\bar{X}_2)_{Y, V}\colon (\bar{X}Y) V \to
\bar{X}(Y V)$ are characterized by commutativity of the diagrams
$$
\begin{tikzcd}
  \cC(U, (\bar{X}Y) V) \ar[rr,"T\mapsto \bar{X}_2T"] \ar[d,"T\mapsto(\iota\otimes\bar{R}_V^*)(T\otimes\iota)"'] & & \cC(U, \bar{X}(Y V)) \ar[d,"\theta"]\\
  \cC(U \bar{V}, \bar{X} Y) \ar[rd,"\theta"'] & & \cD(X U, Y V) \ar[dl,"S\mapsto(\iota\otimes\bar{R}_V^*)(S\otimes\iota)"]\\
  & \cD(X U \bar{V}, Y)
\end{tikzcd}
$$
for $U \in \cC$. From this characterization we have $(\bar{X}_2)_{Y
U, V} ((\bar{X}_2)_{Y, U}\otimes\iota_V) = (\bar{X}_2)_{Y, U V}$ as
morphisms from $(\bar{X}Y) U V$ to $\bar{X}(Y U V)$.

Denote by $\mu_Y \colon X \bar{X} Y \to Y$ the morphism which
corresponds to $\iota_{\bar{X} Y}$ under
isomorphism~\eqref{eq:barX-as-internal-hom} for $U = \bar{X} Y$.
Then the morphism $\bar X(\mu_X)(\bar X_2)_{X,\bar XX}\colon (\bar X
X)(\bar X X)\to\bar XX$ defines an algebra structure on $\bar XX$
with unit $\iota_X\otimes\iota_\un\in
\DD(X)\otimes\CC(\un)\subset\CC(\un,\bar XX)$. Furthermore, the
morphism $\bar X(\mu_Y)(\bar X_2)_{X,\bar XY}\colon (\bar X X)(\bar
X Y)\to\bar XY$ defines a left $\bar{X} X$-module structure on
$\bar{X} Y$. Then the functor $Y \mapsto \bar{X} Y$ extends to an
equivalence between~$\DD$ and the category of left $\bar XX$-modules
in $\CC$ (without any compatibility with the $*$-structures for the
moment)~\citelist{\cite{MR1976459}\cite{MR3242743}*{Section~7.9}}.
It remains to show that $\bar{X} X$ is an irreducible $Q$-system and
that the $\bar{X} X$-module structure on $\bar{X} Y$ is unitary.

\begin{lemma}
The morphism $(\bar{X}_2)_{Y,V}$ is given by
$$
\sum_{i,j} \left(\frac{d_i}{d_j}\right)^{1/2}  F_\beta \otimes v_\beta^*\colon \bigoplus_i \cD(X, Y \bar{U}_i) \otimes U_i V \to \bigoplus_j \cD(X, Y V \bar{U}_j) \otimes U_j,
$$
where $(v_\beta\colon U_j \to U_i V)_\beta$ is an orthonormal basis of isometries, and $F_\beta$ is the map
$$
\cD(X, Y \bar{U}_i)  \to \cD(X, Y V \bar{U}_j), \quad S \mapsto
(\iota \otimes R_i^* \otimes \iota_{V \bar U_{j}}) (S \otimes
v_\beta \otimes \iota_{\bar{\jmath}}) (\iota \otimes \bar{R}_j).
$$
\end{lemma}

\begin{proof}
It is enough to check the commutativity of the above diagram. Let
$$
S \otimes T \in \cD(X, Y \bar{U}_i) \otimes \cC(U, U_i V) \subset
\cC(U, (\bar{X} Y) V).
$$
Chasing this element along the arrows on the left, we obtain
$$
d_i^{1/2} (\iota \otimes R_i^* \otimes \bar{R}_V^*) (S \otimes T \otimes \iota_{\bar{V}}) \in \cD(X U \bar{V}, Y).
$$
On the other hand, chasing the top and right arrows, we obtain
\begin{multline*}
\sum_{j, \beta} d_i^{1/2} (\iota \otimes R_i^* \otimes
\bar{R}_{V}^*) (S \otimes v_\beta \otimes R^*_j(\iota_{\bar{\jmath}}
\otimes v_\beta^* T) \otimes \iota_{\bar{V}}) (\iota_X \otimes
\bar{R}_j \otimes \iota_{U\bar V})\\
=\sum_{j, \beta} d_i^{1/2} (\iota \otimes R_i^* \otimes
\bar{R}_{V}^*) (S \otimes v_\beta v_\beta^* T \otimes
\iota_{\bar{V}})=d_i^{1/2} (\iota \otimes R_i^* \otimes \bar{R}_V^*)
(S \otimes T \otimes \iota_{\bar{V}}),
\end{multline*}
which proves the assertion.
\end{proof}

We can now show that $\bar{X}_2$ is unitary thanks to the normalization of \eqref{eq:barX-as-internal-hom}.

\begin{lemma}
\label{lem:barX-2-is-unitary}
The morphism $\bar{X}_2$ is unitary, and its inverse is given by
$$
\sum_{i,j} \left(\frac{d_i}{d_j}\right)^{1/2}  G_\beta \otimes v_\beta \colon \bigoplus_j \cD(X, Y V \bar{U}_j) \otimes U_j \to \bigoplus_i \cD(X, Y \bar{U}_i) \otimes U_i V,
$$
where $G_\beta$ is the map
$$
\cD(X, Y V \bar{U}_j) \to \cD(X, Y \bar{U}_i), \quad T \mapsto
(\iota_{Y \bar{U}_i} \otimes \bar{R}_j^*) (\iota_{Y \bar{U}_i}
\otimes v_\beta^* \otimes \iota_{\bar{\jmath}}) (\iota_Y \otimes R_i
\otimes \iota_{V \bar{U}_j}) T.
$$
\end{lemma}

\begin{proof}
Since $F_\beta$ can be written as
$$
F_\beta(S) = (\iota \otimes R_i^* \otimes \iota_{V \bar{U}_j}) (\iota_{Y \bar{U}_i} \otimes v_\beta \otimes \iota_{\bar{\jmath}}) (\iota_{Y \bar{U}_i} \otimes \bar{R}_j) S,
$$
the morphism $\bar{X}^*_2$ is indeed given by the formula in the
formulation. It remains to show that $\bar{X}_2$ is an isometry.

By the above formula, the component of $\bar{X}_2^* \bar{X}_2$ for $\cD(X, Y \bar{U}_i) \otimes U_i V \to \cD(X, Y \bar{U}_{i'}) \otimes U_{i'} V$ is given by
$$
\sum_{j, \beta, \gamma} \frac{(d_i d_{i'})^{1/2}}{d_j} H_{\beta, \gamma} \otimes w_{\gamma} v_\beta^*,
$$
where $(w_{\gamma}\colon U_j \to U_{i'} V)_{\gamma}$ is an orthonormal basis, and $H_{\beta,\gamma}$ is the linear map
\begin{gather*}
\cD(X, Y \bar{U}_i) \to \cD(X, Y \bar{U}_{i'}), \\ S \mapsto
(\iota_{Y \bar{U}_{i'}} \otimes \bar{R}_j^*) (\iota_{Y \bar{U}_{i'}}
\otimes w_\gamma^* \otimes \iota_{\bar{\jmath}}) (\iota_Y \otimes R_{i'}
R_i^* \otimes \iota_{V \bar{U}_j}) (\iota_{Y \bar{U}_i} \otimes
v_\beta \otimes \iota_{\bar{\jmath}}) (S \otimes \bar{R}_j).
\end{gather*}
Since $(\iota_{\bar{\imath}'} \otimes \bar{R}_j^*) (\iota_{\bar{\imath}'}
\otimes w_\gamma^* \otimes \iota_{\bar{\jmath}}) (R_{i'} R_i^* \otimes
\iota_{V \bar{U}_j}) (\iota_{\bar{\imath}} \otimes v_\beta \otimes
\iota_{\bar{\jmath}}) (\iota_{\bar{\imath}} \otimes \bar{R}_j)$ is a morphism
from $U_i$ to $U_{i'}$, the only nonzero terms are for $i' = i$.
Moreover, it is easy to see that the family
$$
\left(\left(\frac{d_i}{d_j}\right)^{1/2} (R_i^* \otimes \iota_{V \bar{U}_j}) (\iota_{\bar{\imath}} \otimes v_\beta \otimes \iota_{\bar{\jmath}}) (\iota_{\bar{\imath}} \otimes \bar{R}_j) \right)_\beta
$$
forms an orthonormal basis of isometries $\bar{U}_i \to V
\bar{U}_j$. It follows that if (for $i'=i$) we take
$(w_\gamma)_\gamma=(v_\beta)_\beta$, then $(d_i / d_j)
H_{\beta,\beta}(S) = S$. Thus, we see that $\bar{X}_2^* \bar{X}_2$
indeed acts as the identity morphism on the direct summand $\cD(X, Y
\bar{U}_i) \otimes U_i V$.
\end{proof}

\begin{lemma}
\label{lem:frob-recip-on-mod-and-inn-prod} The linear isomorphism
$$
\cD(X_a, X_b U) \to \cD(X_a \bar{U}, X_b), \quad T \mapsto
\frac{d(\bar XX_b)^{1/2}}{d(\bar XX_a)^{1/2}}(\iota_b \otimes
\bar{R}_U^*) (T \otimes \iota_{\bar{U}}),
$$
is unitary for any $U\in\CC$.
\end{lemma}

Note that by the indecomposability assumption the object $\bar{X}
X_a$ is nonzero for any $a$, so the formulation makes sense.

\begin{proof}
Put $\tilde{T} = (\iota_b \otimes \bar{R}_U^*) (T \otimes
\iota_{\bar{U}})$. By definition of $(\tilde{S}, \tilde{T})$, we
have
$$
\bar{X}((\iota_b \otimes \bar{R}_U^*) (S \otimes
\iota_{\bar{U}}))(T^* \otimes \iota_{\bar{U}}) (\iota_b \otimes
\bar{R}_U)) = (\tilde{S}, \tilde{T}) \iota_{\bar{X} X_b}.
$$
Using the module functor structure on $\bar X$, the left hand side
can be written as
$$
(\iota_{\bar{X} X_b} \otimes \bar{R}_U^*) (\bar{X}_2)_{X_b, U
\bar{U}}^{-1} (\bar{X}_2)_{X_bU, \bar{U}} (\bar{X}(S T^*) \otimes
\iota_{\bar{U}}) (\bar{X}_2)_{X_bU, \bar{U}}^{-1} (\bar{X}_2)_{X_b,
U \bar{U}} (\iota_{\bar{X} X_b} \otimes \bar{R}_U),
$$
which is $(\iota_{\bar{X} X_b} \otimes \bar{R}_U^*)
((\bar{X}_2)_{X_b, U }^{-1}\bar{X}(S T^*) (\bar{X}_2)_{X_b, U }
\otimes \iota_{\bar{U}}) (\iota_{\bar{X} X_b} \otimes \bar{R}_U)$ by
the multiplicativity of $(\bar{X}_2)_{Y,U}$ in $U$. Thus, the scalar
$(\tilde{S}, \tilde{T})$ can be extracted by applying the
categorical trace:
\begin{align*}
(\tilde{S}, \tilde{T}) &= \tr_{\bar{X} X_b} \left( (\iota_{\bar{X} X_b} \otimes \bar{R}_U^*) ((\bar{X}_2)_{X_b, U }^{-1}\bar{X}(S T^*) (\bar{X}_2)_{X_b, U } \otimes \iota_{\bar{U}}) (\iota_{\bar{X} X_b} \otimes \bar{R}_U) \right) \\
&= \frac{1}{d(\bar{X} X_b)} \Tr_{(\bar{X}X_b) U} \left(
(\bar{X}_2)_{X_b, U }^{-1}\bar{X}(S T^*) (\bar{X}_2)_{X_b, U }
\right)\\
&=\frac{1}{d(\bar{X} X_b)} \Tr_{\bar{X} X_a} (\bar{X}(T^*)
\bar{X}(S) ) = \frac{d(\bar{X} X_a)}{d(\bar{X} X_b)} (S, T),
\end{align*}
which proves the assertion.
\end{proof}

\begin{proposition}
\label{prop:mu-Y-is-coisometry} We have $\mu_Y \mu_Y^* = d(\bar{X}
X) \iota_Y$ for any $Y \in \cD$.
\end{proposition}

\begin{proof}
We may assume that $X = X_a$ and $Y = X_b$ for some $a, b$. We
identify
$$
\cC(\bar{X}_a X_b) = \bigoplus_{i,j} B(\cD(X_a, X_b \bar{U}_i),
\cD(X_a, X_b \bar{U}_j)) \otimes \cC(U_i, U_j)
$$
with
$$
\bigoplus_i \cD(X_a, X_b \bar{U}_i) \otimes \overline{\cD(X_a, X_b
\bar{U}_i)} \otimes \cC(U_i,U_j),
$$
so that $\iota_{\bar{X}_a X_b}$ is represented by $\sum_{i,\alpha}
u_\alpha \otimes \bar u_\alpha \otimes \iota_i$, where $(u_\alpha
\colon X_a \to X_b \bar{U}_i)_\alpha$ is an orthonormal basis of
isometries. Then we have
$$
\mu_b = \theta_{X_a, X_b, \bar{X}_a X_b}\left( \sum_{i,\alpha}
u_\alpha \otimes \bar u_\alpha \otimes \iota_i \right ) =
\sum_{i,\alpha} d_i^{1/2} (\iota_b \otimes R_i^*) (u_\alpha \otimes
\iota_{i}) (\iota_X \otimes (u_\alpha^* \otimes \iota_i) p_i),
$$
where $p_i \colon \bar{X}_a X_b \to \cD(X_a, X_b \bar{U}_i) \otimes U_i$ is the orthogonal projection onto the isotypic component for $U_i$. Hence $\mu_b^* \colon X_b \to X_a \bar{X}_a X_b$ is given by $\sum_{i,\alpha} d_i^{1/2} u_\alpha \otimes (u_\alpha^* \otimes \iota_{\bar{\imath}}) (\iota_b \otimes R_i)$. We thus have
$$
\mu_b \mu_b^* = \sum_{i, \alpha, \alpha'} d_i (\iota_b \otimes R_i^*) (u_{\alpha'} u_\alpha^* \otimes \iota_{i}) (\iota_b \otimes R_i).
$$
By Lemma~\ref{lem:frob-recip-on-mod-and-inn-prod}, we have $(\iota_b
\otimes R_i^*) (u_{\alpha'} u_\alpha^* \otimes \iota_{i}) (\iota_b
\otimes R_i) = \delta_{\alpha, \alpha'}\frac{d(\bar X_aX_a)}{d(\bar
X_a X_b)}\iota_b$. Hence
$$
\mu_b\mu_b^*=\sum_i\Big(d_i\frac{d(\bar X_aX_a)}{d(\bar X_a
X_b)}\dim\DD(X_a,X_b\bar U_i)\Big)\iota_b=d(\bar X_a X_a)\iota_b,
$$
which finishes the proof of the proposition.
\end{proof}

It follows that $\bar{X} X$ is a standard $Q$-system in $\cC$ and,
for any $Y \in \cD$, the $\bar{X} X$-module $\bar{X} Y$ satisfies
the unitarity condition. The $Q$-system $\bar XX$ is irreducible,
since $\CC(\un,\bar XX)$ is one-dimensional. Thus
Theorem~\ref{thm:proper-mod-equiv-to-Q-sys-mod} is proved.

\bigskip

\end{document}